\documentclass[12pt,a4paper,reqno]{amsart}
\usepackage{amsfonts}
\usepackage{amssymb}
\usepackage{amsmath}
\usepackage{graphicx}
\usepackage{axodraw}  %Pour les arbres, les graphes, etc.
\usepackage{amsthm}
\usepackage{amscd}
\usepackage[all]{xy}               % pour les diagrammes
\usepackage{epsfig,exscale}  % pour les dessins
\usepackage{color}
\usepackage{txfonts} 
 \font \eightrm=cmr8
 
\newcommand{\nc}{\newcommand}

\newcommand{\lbutcher}{{\raise 3.9pt\hbox{$\circ$}}\hskip -1.9pt{\scriptstyle \searrow}}

\newtheorem{thm}{Theorem}
\newtheorem{exam}{Example}

\newtheorem{cor}[thm]{Corollary}

\newtheorem{lem}[thm]{Lemma}
\newtheorem{prop}[thm]{Proposition}

\newtheorem{rmk}[thm]{Remark}

\nc{\ignore}[1]{{}}
\nc{\mrm}[1]{{\rm #1}}
\nc{\dirlim}{\displaystyle{\lim_{\longrightarrow}}\,}
\nc{\invlim}{\displaystyle{\lim_{\longleftarrow}}\,}
\nc{\vep}{\varepsilon} \nc{\ep}{\epsilon}
\nc{\sigmat}{\widetilde\sigma}
\nc{\ostar}{\overline{*}}

\nc{\mchar}{\mrm{Char}}
\nc{\Hom}{\mrm{Hom}}
\nc{\id}{\mrm{id}}

\nc{\remark}{\noindent{\bf{Remark:}}}
\nc{\remarks}{\noindent{\bf{Remarks:}}}

 \nc{\grad}[1]{^{({#1})}}
 \nc{\fil}[1]{_{#1}}

\nc{\BA}{{\Bbb A}} \nc{\CC}{{\Bbb C}} \nc{\DD}{{\Bbb D}}
\nc{\EE}{{\Bbb E}} \nc{\FF}{{\Bbb F}} \nc{\GG}{{\Bbb G}}
\nc{\HH}{{\Bbb H}} \nc{\LL}{{\Bbb L}} \nc{\NN}{{\Bbb N}}
\nc{\PP}{{\Bbb P}} \nc{\QQ}{{\Bbb Q}} \nc{\RR}{{\Bbb R}}
\nc{\TT}{{\Bbb T}} \nc{\VV}{{\Bbb V}} \nc{\ZZ}{{\Bbb Z}}
\nc{\Cal}[1]{{\mathcal {#1}}}
\nc{\mop}[1]{\mathop{\hbox {\rm #1} }}
\nc{\smop}[1]{\mathop{\hbox {\eightrm #1} }}
\nc{\mopl}[1]{\mathop{\hbox {\rm #1} }\limits}
\nc{\frakg}{{\frak g}}
\nc{\g}[1]{{\frak {#1}}}
\nc{\wt}{\widetilde}
\nc{\wh}{\widehat}
\nc{\un}{\hbox{\bf 1}}
\nc{\redtext}[1]{\textcolor{red}{#1}}
\nc{\bluetext}[1]{\textcolor{blue}{#1}}

\nc\fleche[1]{\mathop{\hbox to #1 mm{\rightarrowfill}}\limits}
\def\semi{\mathrel{\times}\kern -.85pt\joinrel\mathrel{\raise
    1.4pt\hbox{${\scriptscriptstyle |}$}}}

\def\fleche#1{\mathop{\hbox to #1 mm{\rightarrowfill}}\limits}
\def\gfleche#1{\mathop{\hbox to #1 mm{\leftarrowfill}}\limits}
\def\inj#1{\mathop{\hbox to #1 mm{$\lhook\joinrel$\rightarrowfill}}\limits}
\def\ginj#1{\mathop{\hbox to #1 mm{\leftarrowfill$\joinrel\rhook$}}\limits}
\def\surj#1{\mathop{\hbox to #1 mm{\rightarrowfill\hskip 2pt\llap{$\rightarrow$}}}\limits}
\def\gsurj#1{\mathop{\hbox to #1 mm{\rlap{$\leftarrow$}\hskip 2pt
      \leftarrowfill}}\limits}
\def \restr#1{\mathstrut_{\textstyle |}\raise-6pt\hbox{$\scriptstyle #1$}}
\def \srestr#1{\mathstrut_{\scriptstyle |}\hbox to
-1.5pt{}\raise-4pt\hbox{$\scriptscriptstyle #1$}}

\def\diagrama #1{\vskip 4mm \centerline {#1} \vskip 4mm}
\newcommand{\treel}{\hskip 0.5pc\scalebox{-0.3}{{\parbox{0.5pc}{
  \begin{picture}(60,45) (75,10)
    \SetWidth{1.5}
    \SetColor{Black}
    \Line(90,0)(90,45)
  \end{picture}}}}}

\newcommand{\treesmall}{\hskip 0.8pc\scalebox{-0.3}{{\parbox{0.5pc}{
  \begin{picture}(30,45) (75,-30)
    \SetWidth{1.5}
    \SetColor{Black}
   \Line(90,0)(70,-40)
    \Line(90,0)(110,-40)
    \Line(90,0)(90,15)
  \end{picture}}}}}

\newcommand{\treeA}{\hskip 1.5pc\scalebox{-0.3}{{\parbox{0.5pc}{
   \begin{picture}(60,75) (75,-30)
    \SetWidth{1.5}
    \SetColor{Black}
    \Line(90,0)(75,-30)
    \Line(90,0)(105,-30)
    \Line(105,30)(90,0)
    \Line(105,30)(135,-30)
    \Line(105,45)(105,30)
  \end{picture}}}}}

\newcommand{\treeB}{\hskip 1.5pc\scalebox{-0.3}{{\parbox{0.5pc}{
  \begin{picture}(60,75) (75,-30)
    \SetWidth{1.5}
    \SetColor{Black}
    \Line(90,0)(75,-30)
    \Line(105,30)(135,-30)
    \Line(105,45)(105,30)
    \Line(120,0)(105,-30)
    \Line(105,30)(90,0)
  \end{picture}}}}}

\newcommand{\treeD}{\hskip 1.5pc\scalebox{-0.2}{{\parbox{0.5pc}{
 \begin{picture}(90,105) (75,-30)
    \SetWidth{2.5}
    \SetColor{Black}
    \Line(90,0)(75,-30)
    \Line(120,60)(90,0)
    \Line(90,0)(105,-30)
    \Line(120,60)(165,-30)
    \Line(120,80)(120,60)
    \Line(150,0)(135,-30)
  \end{picture}
}}}}

\newcommand{\treeE}{\hskip 1.5pc\scalebox{-0.2}{{\parbox{0.5pc}{
 \begin{picture}(90,105) (75,-30)
    \SetWidth{2.5}
    \SetColor{Black}
    \Line(90,0)(75,-30)
    \Line(120,60)(90,0)
    \Line(120,60)(165,-30)
    \Line(120,80)(120,60)
    \Line(150,0)(135,-30)
    \Line(135,30)(105,-30)
  \end{picture}
}}}}

\newcommand{\treeF}{\hskip 1.5pc\scalebox{-0.2}{{\parbox{0.5pc}{
 \begin{picture}(90,105) (75,-30)
    \SetWidth{2.5}
    \SetColor{Black}
    \Line(90,0)(75,-30)
    \Line(120,60)(90,0)
    \Line(120,60)(165,-30)
    \Line(120,80)(120,60)
    \Line(135,30)(105,-30)
    \Line(120,0)(135,-30)
  \end{picture}
}}}}

\newcommand{\treeG}{\hskip 1.5pc\scalebox{-0.2}{{\parbox{0.5pc}{
 \begin{picture}(90,105) (75,-30)
    \SetWidth{2.5}
    \SetColor{Black}
    \Line(90,0)(75,-30)
    \Line(120,60)(90,0)
    \Line(120,60)(165,-30)
    \Line(120,80)(120,60)
    \Line(105,30)(135,-30)
    \Line(120,0)(105,-30)
  \end{picture}
}}}}

\def\racine{{\scalebox{0.3}{
\begin{picture}(12,12)(38,-38)
\SetWidth{0.5} \SetColor{Black} \Vertex(45,-28){6.5}
\end{picture}}}}

 \def\arbrea{\,{\scalebox{0.15}{ 
  \begin{picture}(8,55) (370,-248)
    \SetWidth{2}
    \SetColor{Black}
    \Line(374,-244)(374,-200)
    \Vertex(374,-197){9}
    \Vertex(375,-245){12}
  \end{picture}
}}\,}

 \def\arbreba{\,{\scalebox{0.15}{ 
\begin{picture}(8,106) (370,-197)
    \SetWidth{2}
    \SetColor{Black}
    \Line(374,-193)(374,-149)
    \Vertex(374,-146){9}
    \Vertex(375,-194){12}
    \Line(374,-142)(374,-98)
    \Vertex(374,-95){9}
  \end{picture}
}}\,}

 \def\arbrebb{\,{\scalebox{0.15}{ 
  \begin{picture}(48,48) (349,-255)
    \SetWidth{2}
    \SetColor{Black}
    \Vertex(375,-252){12}
    \Line(376,-250)(395,-215)
    \Line(373,-251)(354,-214)
    \Vertex(353,-211){9}
    \Vertex(395,-213){9}
  \end{picture}
}}}

\def\arbreca{\,{\scalebox{0.15}{
\begin{picture}(8,156) (370,-147)
    \SetWidth{2}
    \SetColor{Black}
    \Line(374,-143)(374,-99)
    \Vertex(374,-96){9}
    \Vertex(375,-144){12}
    \Line(374,-92)(374,-48)
    \Vertex(374,-45){9}
    \Line(374,-42)(374,2)
    \Vertex(374,5){9}
  \end{picture}
}}\,}

\def\arbrecb{\,{\scalebox{0.15}{
\begin{picture}(48,94) (349,-255)
\SetWidth{2}
\SetColor{Black}
\Line(375,-200)(395,-150)
\Line(375,-200)(354,-150)
\Vertex(354,-150){9}
\Vertex(395,-150){9}
\Vertex(375,-200){9}
\Line(375,-250)(375,-200)
\Vertex(375,-250){12}
\end{picture}}}\,}

\def\arbrecc{\,{\scalebox{0.15}{
 \begin{picture}(48,98) (349,-205)
    \SetWidth{2}
    \SetColor{Black}
    \Vertex(375,-202){12}
    \Line(376,-200)(395,-165)
    \Line(373,-201)(354,-164)
    \Vertex(353,-161){9}
    \Vertex(395,-163){9}
    \Line(353,-160)(353,-113)
    \Vertex(353,-111){9}
  \end{picture}
}}\,}

\def\arbreccc{\,{\scalebox{0.15}{
 \begin{picture}(48,98) (349,-205)
    \SetWidth{2}
    \SetColor{Black}
    \Vertex(375,-202){12}
    \Line(376,-200)(395,-165)
    \Line(373,-201)(354,-164)
    \Vertex(353,-161){9}
    \Vertex(395,-163){9}
    \Line(395,-160)(395,-113)
    \Vertex(395,-111){9}
  \end{picture}
}}\,}

\def\arbreda{\,{\scalebox{0.15}{
\begin{picture}(8,204) (370,-99)
    \SetWidth{2}
    \SetColor{Black}
    \Line(374,-95)(374,-51)
    \Vertex(374,-48){9}
    \Vertex(375,-96){12}
    \Line(374,-44)(374,0)
    \Vertex(374,3){9}
    \Line(374,6)(374,50)
    \Vertex(374,53){9}
    \Line(374,53)(374,98)
    \Vertex(374,101){9}
  \end{picture}
}}\,}

\def\arbredb{\,{\scalebox{0.15}{
\begin{picture}(48,135) (349,-255)
    \SetWidth{2}
    \SetColor{Black}
    \Line(376,-150)(395,-100)
    \Line(373,-150)(354,-100)
    \Vertex(353,-100){9}
    \Vertex(395,-100){9}
    \Vertex(374,-150){9}
    \Line(374,-200)(374,-150)
    \Vertex(374,-200){9}
    \Line(374,-250)(374,-200)
    \Vertex(374,-250){12}
  \end{picture}
}}\,}

\def\arbredc{\,{\scalebox{0.15}{
 \begin{picture}(48,150) (349,-205)
    \SetWidth{2}
    \SetColor{Black}
    \Line(376,-148)(395,-113)
    \Line(373,-149)(354,-112)
    \Vertex(353,-109){9}
    \Vertex(395,-111){9}
    \Line(353,-108)(353,-61)
    \Vertex(353,-59){9}
    \Line(374,-200)(374,-153)
    \Vertex(374,-149){9}
    \Vertex(374,-202){12}
  \end{picture}
}}\,}

\def\arbredcc{\,{\scalebox{0.15}{
 \begin{picture}(48,150) (349,-205)
    \SetWidth{2}
    \SetColor{Black}
    \Line(376,-148)(395,-113)
    \Line(373,-149)(354,-112)
    \Vertex(353,-109){9}
    \Vertex(395,-111){9}
    \Line(395,-108)(395,-61)
    \Vertex(395,-59){9}
    \Line(374,-200)(374,-153)
    \Vertex(374,-149){9}
    \Vertex(374,-202){12}
  \end{picture}
}}\,}

\def\arbredd{\,{\scalebox{0.15}{
 \begin{picture}(48,99) (349,-251)
    \SetWidth{2}
    \SetColor{Black}
    \Line(376,-199)(395,-164)
    \Line(373,-200)(354,-163)
    \Vertex(353,-160){9}
    \Vertex(395,-162){9}
    \Vertex(376,-156){9}
    \Vertex(376,-248){12}
    \Line(375,-245)(375,-204)
    \Line(375,-200)(375,-159)
    \Vertex(375,-201){9}
  \end{picture}
}}\,}

\def\arbrede{\,{\scalebox{0.15}{
 \begin{picture}(48,153) (349,-150)
    \SetWidth{2}
    \SetColor{Black}
    \Vertex(375,-147){12}
    \Line(376,-145)(395,-110)
    \Line(373,-146)(354,-109)
    \Vertex(353,-106){9}
    \Vertex(395,-108){9}
    \Line(353,-105)(353,-58)
    \Vertex(353,-56){9}
    \Line(353,-52)(353,-5)
    \Vertex(353,-1){9}
  \end{picture}
}}\,}

\def\arbredee{\,{\scalebox{0.15}{
 \begin{picture}(48,153) (349,-150)
    \SetWidth{2}
    \SetColor{Black}
    \Vertex(375,-147){12}
    \Line(376,-145)(395,-110)
    \Line(373,-146)(354,-109)
    \Vertex(353,-106){9}
    \Vertex(395,-108){9}
    \Line(395,-105)(395,-58)
    \Vertex(395,-56){9}
    \Line(395,-52)(395,-5)
    \Vertex(395,-1){9}
  \end{picture}
}}\,}

\def\arbredf{\,{\scalebox{0.15}{
\begin{picture}(48,98) (349,-205)
    \SetWidth{2}
    \SetColor{Black}
    \Vertex(375,-202){12}
    \Line(376,-200)(395,-165)
    \Line(373,-201)(354,-164)
    \Vertex(353,-161){9}
    \Vertex(395,-163){9}
    \Line(353,-160)(353,-113)
    \Vertex(353,-111){9}
    \Line(395,-159)(395,-112)
    \Vertex(395,-111){9}
  \end{picture}
}}\,}

\def\arbredz{\,{\scalebox{0.15}{
  \begin{picture}(68,88) (329,-215)
    \SetWidth{2}
    \SetColor{Black}
    \Vertex(375,-212){12}
    \Line(376,-210)(395,-175)
    \Line(373,-211)(354,-174)
    \Vertex(353,-171){9}
    \Vertex(395,-173){9}
    \Line(351,-168)(332,-131)
    \Line(355,-168)(374,-133)
    \Vertex(333,-131){9}
    \Vertex(374,-131){9}
  \end{picture}
}}\,}
\def\arbredzz{\,{\scalebox{0.15}{
  \begin{picture}(68,88) (329,-215)
    \SetWidth{2}
    \SetColor{Black}
    \Vertex(375,-212){12}
    \Line(376,-210)(395,-175)
    \Line(373,-211)(354,-174)
    \Vertex(353,-171){9}
    \Vertex(395,-173){9}
    \Line(393,-168)(377,-131)
    \Line(398,-168)(418,-133)
    \Vertex(418,-131){9}
    \Vertex(377,-131){9}
  \end{picture}
}}\,}

\def\arbredg{\,{\scalebox{0.15}{
\begin{picture}(48,98) (349,-205)
    \SetWidth{2}
    \SetColor{Black}
    \Vertex(375,-202){12}
    \Line(376,-200)(395,-165)
    \Line(373,-201)(354,-164)
    \Vertex(353,-161){9}
    \Vertex(395,-163){9}
    \Line(375,-201)(375,-160)
    \Vertex(376,-157){9}
    \Vertex(353,-111){9}
    \Line(353,-161)(353,-111)
  \end{picture}
}}\,}

\def\arbredh{\,{\scalebox{0.15}{
 \begin{picture}(90,46) (330,-257)
    \SetWidth{2}
    \SetColor{Black}
    \Vertex(375,-254){12}
    \Line(376,-252)(395,-217)
    \Vertex(395,-215){9}
    \Line(374,-254)(335,-226)
    \Vertex(334,-224){9}
    \Line(375,-252)(356,-215)
    \Vertex(355,-215){9}
    \Line(374,-255)(417,-227)
    \Vertex(418,-225){9}
  \end{picture}
}}\,}

\def\arbreddb{\,{\scalebox{0.15}{
 \begin{picture}(48,150) (349,-205)
    \SetWidth{2}
    \SetColor{Black}
    \Line(376,-108)(395,-60)
    \Line(373,-109)(354,-60)
    \Vertex(353,-59){9}
    \Vertex(395,-59){9}
    \Line(353,-60)(353,-10)
    \Vertex(353,-9){9}
    \Line(374,-153)(374,-108)
    \Vertex(374,-109){9}
    \Vertex(374,-153){9}
		\Line(374,-200)(374,-153)
		\Vertex(374,-202){12}
  \end{picture}
}}\,}

\def\arbredddb{\,{\scalebox{0.15}{
 \begin{picture}(48,150) (349,-205)
    \SetWidth{2}
    \SetColor{Black}
    \Line(376,-108)(395,-60)
    \Line(373,-109)(354,-60)
    \Vertex(353,-59){9}
    \Vertex(395,-59){9}
    \Line(395,-60)(395,-10)
    \Vertex(395,-9){9}
    \Line(374,-153)(374,-108)
    \Vertex(374,-109){9}
    \Vertex(374,-153){9}
    \Line(374,-200)(374,-153)
    \Vertex(374,-202){12}
  \end{picture}
}}\,}

\def\arbreddd{\,{\scalebox{0.15}{
 \begin{picture}(48,153) (349,-150)
    \SetWidth{2}
    \SetColor{Black}
    \Vertex(375,-147){12}
    \Line(376,-145)(395,-110)
    \Line(373,-146)(354,-109)
    \Vertex(353,-106){9}
    \Vertex(395,-108){9}
    \Line(353,-105)(353,-58)
    \Vertex(353,-56){9}
    \Line(353,-52)(353,-5)
    \Vertex(353,-1){9}
		 \Line(353,-3)(353,50)
    \Vertex(353,49){9}
  \end{picture}
}}\,}

\def\arbredde{\,{\scalebox{0.15}{
\begin{picture}(48,98) (349,-205)
    \SetWidth{2}
    \SetColor{Black}
		\Vertex(375,-202){12}
    \Line(376,-200)(375,-160)
    \Vertex(375,-161){9}
    \Line(376,-160)(395,-115)
    \Line(373,-160)(354,-114)
    \Vertex(353,-111){9}
    \Vertex(395,-113){9}
    \Line(353,-111)(353,-63)
    \Vertex(353,-61){9}
    \Line(395,-111)(395,-62)
    \Vertex(395,-61){9}
  \end{picture}
}}\,}

\def\arbreedb{\,{\scalebox{0.15}{
 \begin{picture}(48,98) (349,-205)
    \SetWidth{2}
    \SetColor{Black}
    \Vertex(375,-197){12}
    \Line(375,-195)(374,-150)
    \Vertex(375,-147){9}
    \Line(376,-145)(395,-110)
    \Line(373,-146)(354,-109)
    \Vertex(353,-106){9}
    \Vertex(395,-108){9}
    \Line(353,-105)(353,-58)
    \Vertex(353,-56){9}
    \Line(353,-52)(353,-5)
    \Vertex(353,-1){9}
  \end{picture}
}}\,}

\def\arbreedf{\,{\scalebox{0.15}{
  \begin{picture}(68,88) (329,-215)
    \SetWidth{2}
    \SetColor{Black}
    \Vertex(375,-212){12}
    \Line(376,-210)(395,-175)
    \Line(373,-211)(354,-174)
    \Vertex(353,-171){9}
    \Vertex(395,-173){9}
    \Line(351,-168)(332,-131)
    \Line(355,-168)(374,-133)
    \Vertex(333,-131){9}
    \Vertex(374,-131){9}
    \Line(334,-131)(333,-85)
    \Vertex(333,-85){9}
  \end{picture}
}}\,}

\def\arbreedff{\,{\scalebox{0.15}{
  \begin{picture}(68,88) (329,-215)
    \SetWidth{2}
    \SetColor{Black}
    \Vertex(375,-212){12}
    \Line(376,-210)(395,-175)
    \Line(373,-211)(354,-174)
    \Vertex(353,-171){9}
    \Vertex(395,-173){9}
    \Line(395,-168)(415,-131)
    \Line(393,-168)(374,-133)
    \Vertex(415,-131){9}
    \Vertex(374,-131){9}
    \Line(374,-131)(374,-85)
    \Vertex(374,-85){9}
  \end{picture}
}}\,}

\def\arbreedg{\,{\scalebox{0.15}{
\begin{picture}(48,98) (349,-205)
    \SetWidth{2}
    \SetColor{Black}
    \Vertex(375,-202){12}
    \Line(376,-200)(395,-165)
    \Line(373,-201)(354,-164)
    \Vertex(353,-161){9}
    \Vertex(395,-163){9}
    \Line(353,-160)(353,-113)
    \Vertex(353,-111){9}
    \Line(395,-159)(395,-112)
    \Vertex(395,-111){9}
    \Line(353,-60)(353,-112)
    \Vertex(353,-60){9}
  \end{picture}
}}\,}

\def\arbreeddh{\,{\scalebox{0.15}{
  \begin{picture}(68,88) (329,-215)
    \SetWidth{2}
    \SetColor{Black}
    \Vertex(375,-212){12}
    \Line(376,-210)(400,-175)
    \Line(373,-211)(350,-174)
    \Vertex(350,-171){9}
    \Vertex(400,-173){9}
    \Line(347,-131)(330,-93)
    \Line(353,-131)(365,-93)
    \Vertex(330,-91){9}
    \Vertex(365,-91){9}
    \Line(350,-168)(350,-131)
    \Vertex(350,-131){9}
		\end{picture}
}}\,}

\def\arbreeddhh{\,{\scalebox{0.15}{
  \begin{picture}(68,88) (329,-215)
    \SetWidth{2}
    \SetColor{Black}
    \Vertex(375,-212){12}
    \Line(376,-175)(375,-120)
    \Line(375,-211)(375,-164)
    \Vertex(375,-165){9}
    \Vertex(375,-120){9}
    \Line(375,-75)(354,-15)
    \Line(375,-75)(394,-15)
    \Vertex(353,-15){9}
    \Vertex(395,-15){9}
    \Line(375,-131)(375,-75)
    \Vertex(375,-75){9}
\end{picture}
}}\,}

\def\arbreedddh{\,{\scalebox{0.15}{
  \begin{picture}(68,88) (329,-215)
    \SetWidth{2}
    \SetColor{Black}
    \Vertex(375,-212){12}
    \Line(376,-210)(400,-175)
    \Line(373,-211)(350,-174)
    \Vertex(350,-171){9}
    \Vertex(400,-173){9}
    \Line(400,-131)(418,-93)
    \Line(400,-131)(385,-93)
    \Vertex(418,-91){9}
    \Vertex(385,-91){9}
    \Line(400,-168)(400,-131)
    \Vertex(400,-131){9}
\end{picture}
}}\,}

\def\arbreea{\,{\scalebox{0.15}{
 \begin{picture}(8,251) (370,-52)
    \SetWidth{2}
    \SetColor{Black}
    \Line(374,-48)(374,-4)
    \Vertex(374,-1){9}
    \Vertex(375,-49){12}
    \Line(374,3)(374,47)
    \Vertex(374,50){9}
    \Line(374,53)(374,97)
    \Vertex(374,100){9}
    \Vertex(374,148){9}
    \Line(374,149)(374,194)
    \Line(374,100)(374,144)
    \Vertex(374,195){9}
  \end{picture}
}}\,}

\def\arbreeb{\,{\scalebox{0.15}{
 \begin{picture}(48,153) (349,-150)
    \SetWidth{2}
    \SetColor{Black}
    \Vertex(375,-197){12}
    \Line(375,-195)(374,-150)
    \Vertex(375,-147){9}
    \Line(376,-145)(395,-110)
    \Line(373,-146)(354,-109)
    \Vertex(353,-106){9}
    \Vertex(395,-108){9}
    \Line(395,-105)(395,-58)
    \Vertex(395,-56){9}
    \Line(395,-52)(395,-5)
    \Vertex(395,-1){9}
  \end{picture}
}}\,}

\def\arbreec{\,{\scalebox{0.15}{
  \begin{picture}(68,88) (329,-215)
    \SetWidth{2}
    \SetColor{Black}
    \Vertex(375,-262){12}
    \Line(375,-263)(374,-205)
    \Vertex(375,-212){9}
    \Line(376,-210)(395,-175)
    \Line(373,-211)(354,-174)
    \Vertex(353,-171){9}
    \Vertex(395,-173){9}
    \Line(351,-168)(332,-131)
    \Line(355,-168)(374,-133)
    \Vertex(333,-131){9}
    \Vertex(374,-131){9}
  \end{picture}
}}\,}
\def\arbreecc{\,{\scalebox{0.15}{
  \begin{picture}(68,88) (329,-215)
    \SetWidth{2}
    \SetColor{Black}
    \Vertex(375,-262){12}
    \Line(375,-263)(374,-205)
    \Vertex(375,-212){9}
    \Line(376,-210)(395,-175)
    \Line(373,-211)(354,-174)
    \Vertex(353,-171){9}
    \Vertex(395,-173){9}
    \Line(393,-168)(377,-131)
    \Line(398,-168)(418,-133)
    \Vertex(418,-131){9}
    \Vertex(377,-131){9}
  \end{picture}
}}\,}
\def\arbreed{\,{\scalebox{0.15}{
 \begin{picture}(90,46) (330,-257)
    \SetWidth{2}
    \SetColor{Black}
    \Vertex(375,-304){12}
    \Line(375,-303)(374,-250)
    \Vertex(375,-254){9}
    \Line(376,-252)(395,-217)
    \Vertex(395,-215){9}
    \Line(374,-254)(335,-226)
    \Vertex(334,-224){9}
    \Line(375,-252)(356,-215)
    \Vertex(355,-215){9}
    \Line(374,-255)(417,-227)
    \Vertex(418,-225){9}
  \end{picture}
}}\,}
\def\arbreee{\,{\scalebox{0.15}{
 \begin{picture}(48,153) (349,-150)
    \SetWidth{2}
    \SetColor{Black}
    \Vertex(375,-147){12}
    \Line(376,-145)(395,-110)
    \Line(373,-146)(354,-109)
    \Vertex(353,-106){9}
    \Vertex(395,-108){9}
    \Line(395,-105)(395,-58)
    \Vertex(395,-56){9}
    \Line(395,-52)(395,-5)
    \Vertex(395,-1){9}
    \Line(395,-52)(395,48)
    \Vertex(395,50){9}
  \end{picture}
}}\,}
\def\arbreef{\,{\scalebox{0.15}{
  \begin{picture}(68,88) (329,-215)
    \SetWidth{2}
    \SetColor{Black}
    \Vertex(375,-212){12}
    \Line(376,-210)(395,-175)
    \Line(373,-211)(354,-174)
    \Vertex(353,-171){9}
    \Vertex(395,-173){9}
    \Line(351,-168)(332,-131)
    \Line(355,-168)(374,-133)
    \Vertex(333,-131){9}
    \Vertex(374,-131){9}
    \Line(374,-131)(374,-85)
    \Vertex(374,-85){9}
  \end{picture}
}}\,}
\def\arbreeff{\,{\scalebox{0.15}{
  \begin{picture}(68,88) (329,-215)
    \SetWidth{2}
    \SetColor{Black}
    \Vertex(375,-212){12}
    \Line(376,-210)(395,-175)
    \Line(373,-211)(354,-174)
    \Vertex(353,-171){9}
    \Vertex(395,-173){9}
    \Line(393,-168)(377,-131)
    \Line(398,-168)(418,-133)
    \Vertex(418,-131){9}
    \Vertex(377,-131){9}
    \Line(418,-131)(418,-85)
    \Vertex(418,-85){9}
  \end{picture}
}}\,}

\def\arbreeg{\,{\scalebox{0.15}{
\begin{picture}(48,98) (349,-205)
    \SetWidth{2}
    \SetColor{Black}
    \Vertex(375,-202){12}
    \Line(376,-200)(395,-165)
    \Line(373,-201)(354,-164)
    \Vertex(353,-161){9}
    \Vertex(395,-163){9}
    \Line(353,-160)(353,-113)
    \Vertex(353,-111){9}
    \Line(395,-159)(395,-112)
    \Vertex(395,-111){9}
    \Line(395,-60)(395,-112)
    \Vertex(395,-60){9}
  \end{picture}
}}\,}
\def\arbreeh{\,{\scalebox{0.15}{
  \begin{picture}(68,88) (329,-215)
    \SetWidth{2}
    \SetColor{Black}
    \Vertex(375,-212){12}
    \Line(376,-210)(400,-175)
    \Line(373,-211)(350,-174)
    \Vertex(350,-171){9}
    \Vertex(400,-173){9}
    \Line(347,-168)(330,-131)
    \Line(353,-168)(365,-133)
    \Vertex(330,-131){9}
    \Vertex(365,-131){9}
    \Line(400,-168)(400,-131)
    \Vertex(400,-131){9}

  \end{picture}
}}\,}
\def\arbreehh{\,{\scalebox{0.15}{
  \begin{picture}(68,88) (329,-215)
    \SetWidth{2}
    \SetColor{Black}
    \Vertex(375,-212){12}
    \Line(376,-210)(400,-175)
    \Line(373,-211)(350,-174)
    \Vertex(350,-171){9}
    \Vertex(400,-173){9}
    \Line(398,-168)(385,-131)
    \Line(405,-168)(422,-133)
    \Vertex(422,-131){9}
    \Vertex(385,-131){9}
    \Line(350,-168)(350,-131)
    \Vertex(350,-131){9}
  \end{picture}
}}\,}
\def\arbreei{\,{\scalebox{0.15}{
 \begin{picture}(90,46) (330,-257)
    \SetWidth{2}
    \SetColor{Black}
    \Vertex(375,-254){12}
    \Line(376,-252)(395,-217)

    \Vertex(395,-215){9}
    \Line(374,-254)(335,-226)
    \Vertex(334,-224){9}
    \Line(375,-252)(356,-215)
    \Vertex(355,-215){9}
    \Line(374,-255)(417,-227)
    \Vertex(418,-225){9}
    \Line(333,-225)(333,-180)
    \Vertex(334,-180){9}
  \end{picture}
}}\,}
\def\arbreeii{\,{\scalebox{0.15}{
 \begin{picture}(90,46) (330,-257)
    \SetWidth{2}
    \SetColor{Black}
    \Vertex(375,-254){12}
    \Line(376,-252)(395,-217)
    \Vertex(395,-215){9}
    \Line(374,-254)(335,-226)
    \Vertex(334,-224){9}
    \Line(375,-252)(356,-215)
    \Vertex(355,-215){9}
    \Line(374,-255)(417,-227)
    \Vertex(418,-225){9}
    \Line(355,-220)(355,-170)
    \Vertex(355,-170){9}
  \end{picture}
}}\,}
\def\arbreeiii{\,{\scalebox{0.15}{
 \begin{picture}(90,46) (330,-257)
    \SetWidth{2}
    \SetColor{Black}
    \Vertex(375,-254){12}
    \Line(376,-252)(395,-217)
    \Vertex(395,-215){9}
    \Line(374,-254)(335,-226)
    \Vertex(334,-224){9}
    \Line(375,-252)(356,-215)
    \Vertex(355,-215){9}
    \Line(374,-255)(417,-227)
    \Vertex(418,-225){9}
    \Line(395,-220)(395,-170)
    \Vertex(395,-170){9}
  \end{picture}
}}\,}
\def\arbreeiiii{\,{\scalebox{0.15}{

 \begin{picture}(90,46) (330,-257)
    \SetWidth{2}
    \SetColor{Black}
    \Vertex(375,-254){12}
    \Line(376,-252)(395,-217)
    \Vertex(395,-215){9}
    \Line(374,-254)(335,-226)
    \Vertex(334,-224){9}
    \Line(375,-252)(356,-215)
    \Vertex(355,-215){9}
    \Line(374,-255)(417,-227)
    \Vertex(418,-225){9}
    \Line(418,-225)(418,-180)
    \Vertex(418,-180){9}
  \end{picture}
}}\,}

\def\arbrefa{\,{\scalebox{0.15}{
 \begin{picture}(8,251) (370,-52)
    \SetWidth{2}
    \SetColor{Black}
    \Vertex(375,-99){12}
    \Line(375,3)(375,-99)
    \Line(375,-48)(375,-4)
    \Vertex(375,-1){9}
    \Vertex(375,-49){9}
    \Line(375,3)(375,47)
    \Vertex(375,50){9}
    \Line(375,53)(375,97)
    \Vertex(375,100){9}
    \Vertex(375,148){9}
    \Line(375,149)(375,194)
    \Line(375,100)(375,144)
    \Vertex(375,195){9}
  \end{picture}
}}\,}
\def\arbrefb{\,{\scalebox{0.15}{
 \begin{picture}(48,153) (349,-150)
    \SetWidth{2}
    \SetColor{Black}
    \Vertex(375,-247){12}
    \Line(375,-195)(375,-247)
    \Vertex(375,-197){9}
    \Line(375,-195)(374,-150)
    \Vertex(375,-147){9}
    \Line(376,-145)(395,-110)
    \Line(373,-146)(354,-109)
    \Vertex(353,-106){9}
    \Vertex(395,-108){9}
    \Line(395,-105)(395,-58)
    \Vertex(395,-56){9}
    \Line(395,-52)(395,-5)
    \Vertex(395,-1){9}
  \end{picture}
}}\,}
\def\arbrefc{\,{\scalebox{0.15}{
  \begin{picture}(68,88) (329,-215)
    \SetWidth{2}
    \SetColor{Black}
    \Vertex(375,-312){12}
    \Line(375,-263)(375,-312)
    \Vertex(375,-262){9}
    \Line(375,-263)(374,-205)
    \Vertex(375,-212){9}
    \Line(376,-210)(395,-175)
    \Line(373,-211)(354,-174)
    \Vertex(353,-171){9}
    \Vertex(395,-173){9}
    \Line(351,-168)(332,-131)
    \Line(355,-168)(374,-133)
    \Vertex(333,-131){9}
    \Vertex(374,-131){9}
  \end{picture}
}}\,}
\def\arbrefcc{\,{\scalebox{0.15}{
  \begin{picture}(68,88) (329,-215)
    \SetWidth{2}
    \SetColor{Black}
    \Vertex(375,-312){12}
    \Line(375,-263)(375,-312)
    \Vertex(375,-262){9}
    \Line(375,-263)(374,-205)
    \Vertex(375,-212){9}
    \Line(376,-210)(395,-175)
    \Line(373,-211)(354,-174)
    \Vertex(353,-171){9}
    \Vertex(395,-173){9}
    \Line(393,-168)(377,-131)
    \Line(398,-168)(418,-133)
    \Vertex(418,-131){9}
    \Vertex(377,-131){9}
  \end{picture}
}}\,}
\def\arbrefd{\,{\scalebox{0.15}{
 \begin{picture}(90,46) (330,-257)
    \SetWidth{2}
    \SetColor{Black}
    \Vertex(375,-354){12}
    \Line(375,-303)(375,-354)
    \Vertex(375,-304){9}
    \Line(375,-303)(374,-250)
    \Vertex(375,-254){9}
    \Line(376,-252)(395,-217)
    \Vertex(395,-215){9}
    \Line(374,-254)(335,-226)
    \Vertex(334,-224){9}
    \Line(375,-252)(356,-215)
    \Vertex(355,-215){9}
    \Line(374,-255)(417,-227)
    \Vertex(418,-225){9}
  \end{picture}
}}\,}
\def\arbrefe{\,{\scalebox{0.15}{
 \begin{picture}(48,153) (349,-150)
    \SetWidth{2}
    \SetColor{Black}
    \Vertex(375,-197){12}
    \Line(375,-145)(375,-197)
    \Vertex(375,-147){9}
    \Line(376,-145)(395,-110)
    \Line(373,-146)(354,-109)
    \Vertex(353,-106){9}
    \Vertex(395,-108){9}
    \Line(395,-105)(395,-58)
    \Vertex(395,-56){9}
    \Line(395,-52)(395,-5)
    \Vertex(395,-1){9}
    \Line(395,-52)(395,48)
    \Vertex(395,50){9}
  \end{picture}
}}\,}
\def\arbreff{\,{\scalebox{0.15}{
  \begin{picture}(68,88) (329,-215)
    \SetWidth{2}
    \SetColor{Black}
    \Vertex(375,-262){12}
    \Line(375,-210)(375,-262)
    \Vertex(375,-212){9}
    \Line(376,-210)(395,-175)
    \Line(373,-211)(354,-174)
    \Vertex(353,-171){9}
    \Vertex(395,-173){9}
    \Line(351,-168)(332,-131)
    \Line(355,-168)(374,-133)
    \Vertex(333,-131){9}
    \Vertex(374,-131){9}
    \Line(374,-131)(374,-85)
    \Vertex(374,-85){9}
  \end{picture}
}}\,}
\def\arbrefff{\,{\scalebox{0.15}{
  \begin{picture}(68,88) (329,-215)
    \SetWidth{2}
    \SetColor{Black}
    \Vertex(375,-262){12}
    \Line(375,-210)(375,-262)
    \Vertex(375,-212){9}
    \Line(376,-210)(395,-175)
    \Line(373,-211)(354,-174)
    \Vertex(353,-171){9}
    \Vertex(395,-173){9}
    \Line(393,-168)(377,-131)
    \Line(398,-168)(418,-133)
    \Vertex(418,-131){9}
    \Vertex(377,-131){9}
    \Line(418,-131)(418,-85)
    \Vertex(418,-85){9}
  \end{picture}
}}\,}

\def\arbrefg{\,{\scalebox{0.15}{
\begin{picture}(48,98) (349,-205)
    \SetWidth{2}
    \SetColor{Black}
    \Vertex(375,-252){12}
    \Line(375,-200)(375,-252)
    \Vertex(375,-202){9}
    \Line(376,-200)(395,-165)
    \Line(373,-201)(354,-164)
    \Vertex(353,-161){9}
    \Vertex(395,-163){9}
    \Line(353,-160)(353,-113)
    \Vertex(353,-111){9}
    \Line(395,-159)(395,-112)
    \Vertex(395,-111){9}
    \Line(395,-60)(395,-112)
    \Vertex(395,-60){9}
  \end{picture}
}}\,}

\def\arbrefh{\,{\scalebox{0.15}{
  \begin{picture}(68,88) (329,-215)
    \SetWidth{2}
    \SetColor{Black}
    \Vertex(375,-262){12}
    \Line(375,-210)(375,-262)
    \Vertex(375,-212){9}
    \Line(376,-210)(400,-175)
    \Line(373,-211)(350,-174)
    \Vertex(350,-171){9}
    \Vertex(400,-173){9}
    \Line(347,-168)(330,-131)
    \Line(353,-168)(365,-133)
    \Vertex(330,-131){9}
    \Vertex(365,-131){9}
    \Line(400,-168)(400,-131)
    \Vertex(400,-131){9}

  \end{picture}
}}\,}
\def\arbrefhh{\,{\scalebox{0.15}{
  \begin{picture}(68,88) (329,-215)
    \SetWidth{2}
    \SetColor{Black}
    \Vertex(375,-262){12}
    \Line(375,-210)(375,-262)
    \Vertex(375,-212){9}
    \Line(376,-210)(400,-175)
    \Line(373,-211)(350,-174)
    \Vertex(350,-171){9}
    \Vertex(400,-173){9}
    \Line(398,-168)(385,-131)
    \Line(405,-168)(422,-133)
    \Vertex(422,-131){9}
    \Vertex(385,-131){9}
    \Line(350,-168)(350,-131)
    \Vertex(350,-131){9}
  \end{picture}
}}\,}

\def\arbrefi{\,{\scalebox{0.15}{
 \begin{picture}(90,46) (330,-257)
    \SetWidth{2}
    \SetColor{Black}
    \Vertex(375,-304){12}
    \Line(375,-252)(375,-304)
    \Vertex(375,-254){9}
    \Line(376,-252)(395,-217)
    \Vertex(395,-215){9}
    \Line(374,-254)(335,-226)
    \Vertex(334,-224){9}
    \Line(375,-252)(356,-215)
    \Vertex(355,-215){9}
    \Line(374,-255)(417,-227)
    \Vertex(418,-225){9}
    \Line(333,-225)(333,-180)
    \Vertex(334,-180){9}
  \end{picture}
}}\,}
\def\arbrefii{\,{\scalebox{0.15}{
 \begin{picture}(90,46) (330,-257)
    \SetWidth{2}
    \SetColor{Black}
    \Vertex(375,-304){12}
    \Line(375,-252)(375,-304)
    \Vertex(375,-254){9}
    \Line(376,-252)(395,-217)
    \Vertex(395,-215){9}
    \Line(374,-254)(335,-226)
    \Vertex(334,-224){9}
    \Line(375,-252)(356,-215)
    \Vertex(355,-215){9}
    \Line(374,-255)(417,-227)
    \Vertex(418,-225){9}
    \Line(355,-220)(355,-170)
    \Vertex(355,-170){9}
  \end{picture}
}}\,}
\def\arbrefiii{\,{\scalebox{0.15}{
 \begin{picture}(90,46) (330,-257)
    \SetWidth{2}
    \SetColor{Black}
    \Vertex(375,-304){12}
    \Line(375,-252)(375,-304)
    \Vertex(375,-254){9}
    \Line(376,-252)(395,-217)
    \Vertex(395,-215){9}
    \Line(374,-254)(335,-226)
    \Vertex(334,-224){9}
    \Line(375,-252)(356,-215)
    \Vertex(355,-215){9}
    \Line(374,-255)(417,-227)
    \Vertex(418,-225){9}
    \Line(395,-220)(395,-170)
    \Vertex(395,-170){9}
  \end{picture}
}}\,}
\def\arbrefiiii{\,{\scalebox{0.15}{
 \begin{picture}(90,46) (330,-257)
    \SetWidth{2}
    \SetColor{Black}
    \Vertex(375,-304){12}
    \Line(375,-252)(375,-304)
    \Vertex(375,-254){9}
    \Line(376,-252)(395,-217)
    \Vertex(395,-215){9}
    \Line(374,-254)(335,-226)
    \Vertex(334,-224){9}
    \Line(375,-252)(356,-215)
    \Vertex(355,-215){9}
    \Line(374,-255)(417,-227)
    \Vertex(418,-225){9}
    \Line(418,-225)(418,-180)
    \Vertex(418,-180){9}
  \end{picture}
}}\,}

\def\arbrefj{\,{\scalebox{0.15}{
\begin{picture}(8,204) (370,-99)
    \SetWidth{2}
    \SetColor{Black}
    \Line(374,-95)(374,-51)
    \Vertex(374,-48){9}
    \Vertex(375,-96){9}
    \Line(374,-44)(374,0)
    \Vertex(374,3){9}
    \Line(374,6)(374,50)
    \Vertex(374,53){9}
    \Line(374,53)(374,98)
    \Vertex(374,101){9}
    \Vertex(350,-146){12}
    \Line(375,-96)(350,-146)
    \Vertex(325,-96){9}
    \Line(325,-96)(350,-146)
  \end{picture}
}}\,}

\def\arbrefk{\,{\scalebox{0.15}{
 \begin{picture}(48,150) (349,-205)
    \SetWidth{2}
    \SetColor{Black}
    \Line(376,-148)(395,-113)
    \Line(373,-149)(354,-112)
    \Vertex(353,-109){9}
    \Vertex(395,-111){9}
    \Line(395,-108)(395,-61)
    \Vertex(395,-59){9}
    \Line(374,-200)(374,-153)
    \Vertex(374,-149){9}
    \Vertex(374,-202){9}
    \Vertex(398,-252){12}
    \Line(374,-202)(398,-252)
    \Vertex(422,-202){9}
    \Line(398,-252)(422,-202)
  \end{picture}
}}\,}

\def\arbrefkk{\,{\scalebox{0.15}{
 \begin{picture}(48,150) (349,-205)
    \SetWidth{2}
    \SetColor{Black}
    \Line(376,-148)(395,-113)
    \Line(373,-149)(354,-112)
    \Vertex(353,-109){9}
    \Vertex(395,-111){9}
    \Line(395,-108)(395,-61)
    \Vertex(395,-59){9}
    \Line(374,-200)(374,-153)
    \Vertex(374,-149){9}
    \Vertex(374,-202){9}
    \Vertex(350,-252){12}
    \Line(374,-202)(350,-252)
    \Vertex(326,-202){9}
    \Line(350,-252)(326,-202)
  \end{picture}
}}\,}

\def\arbrefl{\,{\scalebox{0.15}{
 \begin{picture}(48,153) (349,-150)
    \SetWidth{2}
    \SetColor{Black}
    \Vertex(375,-147){9}
    \Line(376,-145)(395,-110)
    \Line(373,-146)(354,-109)
    \Vertex(353,-106){9}
    \Vertex(395,-108){9}
    \Line(395,-105)(395,-58)
    \Vertex(395,-56){9}
    \Line(395,-52)(395,-5)
    \Vertex(395,-1){9}
    \Vertex(400,-197){12}
    \Line(375,-147)(400,-197)
    \Vertex(425,-147){9}
    \Line(400,-197)(425,-147)
  \end{picture}
}}\,}

\def\arbrefll{\,{\scalebox{0.15}{
 \begin{picture}(48,153) (349,-150)
    \SetWidth{2}
    \SetColor{Black}
    \Vertex(375,-147){9}
    \Line(376,-145)(395,-110)
    \Line(373,-146)(354,-109)
    \Vertex(353,-106){9}
    \Vertex(395,-108){9}
    \Line(395,-105)(395,-58)
    \Vertex(395,-56){9}
    \Line(395,-52)(395,-5)
    \Vertex(395,-1){9}
    \Vertex(350,-197){12}
    \Line(375,-147)(350,-197)
    \Vertex(325,-147){9}
    \Line(350,-197)(325,-147)
  \end{picture}
}}\,}

\def\arbrefm{\,{\scalebox{0.15}{
\begin{picture}(48,98) (349,-205)
    \SetWidth{2}
    \SetColor{Black}
    \Vertex(375,-202){9}
    \Line(376,-200)(395,-165)
    \Line(373,-201)(354,-164)
    \Vertex(353,-161){9}
    \Vertex(395,-163){9}
    \Line(353,-160)(353,-113)
    \Vertex(353,-111){9}
    \Line(395,-159)(395,-112)
    \Vertex(395,-111){9}
    \Vertex(400,-252){12}
    \Line(375,-202)(400,-252)
    \Vertex(425,-202){9}
    \Line(400,-252)(425,-202)
  \end{picture}
}}\,}

\def\arbrefmm{\,{\scalebox{0.15}{
\begin{picture}(48,98) (349,-205)
    \SetWidth{2}
    \SetColor{Black}
    \Vertex(375,-202){9}
    \Line(376,-200)(395,-165)
    \Line(373,-201)(354,-164)
    \Vertex(353,-161){9}
    \Vertex(395,-163){9}
    \Line(353,-160)(353,-113)
    \Vertex(353,-111){9}
    \Line(395,-159)(395,-112)
    \Vertex(395,-111){9}
    \Vertex(350,-252){12}
    \Line(375,-202)(350,-252)
    \Vertex(325,-202){9}
    \Line(350,-252)(325,-202)
  \end{picture}
}}\,}

\def\arbrefn{\,{\scalebox{0.15}{
  \begin{picture}(68,88) (329,-215)
    \SetWidth{2}
    \SetColor{Black}
    \Vertex(375,-212){9}
    \Line(376,-210)(395,-175)
    \Line(373,-211)(354,-174)
    \Vertex(353,-171){9}
    \Vertex(395,-173){9}
    \Line(351,-168)(332,-131)
    \Line(355,-168)(374,-133)
    \Vertex(333,-131){9}
    \Vertex(374,-131){9}
    \Vertex(400,-262){12}
    \Line(375,-212)(400,-262)
    \Vertex(425,-212){9}
    \Line(400,-262)(425,-212)
  \end{picture}
}}\,}

\def\arbrefnn{\,{\scalebox{0.15}{
  \begin{picture}(68,88) (329,-215)
    \SetWidth{2}
    \SetColor{Black}
    \Vertex(375,-212){9}
    \Line(376,-210)(395,-175)
    \Line(373,-211)(354,-174)
    \Vertex(353,-171){9}
    \Vertex(395,-173){9}
    \Line(351,-168)(332,-131)
    \Line(355,-168)(374,-133)
    \Vertex(333,-131){9}
    \Vertex(374,-131){9}
    \Vertex(350,-262){12}
    \Line(375,-212)(350,-262)
    \Vertex(325,-212){9}
    \Line(350,-262)(325,-212)
  \end{picture}
}}\,}

\def\arbrefnnn{\,{\scalebox{0.15}{
  \begin{picture}(68,88) (329,-215)
    \SetWidth{2}
    \SetColor{Black}
    \Vertex(375,-212){9}
    \Line(376,-210)(395,-175)
    \Line(373,-211)(354,-174)
    \Vertex(353,-171){9}
    \Vertex(395,-173){9}
    \Line(393,-168)(377,-131)
    \Line(398,-168)(418,-133)
    \Vertex(418,-131){9}
    \Vertex(377,-131){9}
    \Vertex(400,-262){12}
    \Line(375,-212)(400,-262)
    \Vertex(425,-212){9}
    \Line(400,-262)(425,-212)
  \end{picture}
}}\,}

\def\arbrefnnnn{\,{\scalebox{0.15}{
  \begin{picture}(68,88) (329,-215)
    \SetWidth{2}
    \SetColor{Black}
    \Vertex(375,-212){9}
    \Line(376,-210)(395,-175)
    \Line(373,-211)(354,-174)
    \Vertex(353,-171){9}
    \Vertex(395,-173){9}
    \Line(393,-168)(377,-131)
    \Line(398,-168)(418,-133)
    \Vertex(418,-131){9}
    \Vertex(377,-131){9}
    \Vertex(350,-262){12}
    \Line(375,-212)(350,-262)
    \Vertex(325,-212){9}
    \Line(350,-262)(325,-212)
  \end{picture}
}}\,}

\def\arbrefo{\,{\scalebox{0.15}{
 \begin{picture}(90,46) (330,-257)
    \SetWidth{2}
    \SetColor{Black}
    \Vertex(375,-254){9}
    \Line(376,-252)(395,-217)
    \Vertex(395,-215){9}
    \Line(374,-254)(335,-226)
    \Vertex(334,-224){9}
    \Line(375,-252)(356,-215)
    \Vertex(355,-215){9}
    \Line(374,-255)(417,-227)
    \Vertex(418,-225){9}
    \Vertex(400,-304){12}
    \Line(375,-254)(400,-304)
    \Vertex(425,-254){9}
    \Line(400,-304)(425,-254)
  \end{picture}
}}\,}

\def\arbrefoo{\,{\scalebox{0.15}{
 \begin{picture}(90,46) (330,-257)
    \SetWidth{2}
    \SetColor{Black}
    \Vertex(375,-254){9}
    \Line(376,-252)(395,-217)
    \Vertex(395,-215){9}
    \Line(374,-254)(335,-226)
    \Vertex(334,-224){9}
    \Line(375,-252)(356,-215)
    \Vertex(355,-215){9}
    \Line(374,-255)(417,-227)
    \Vertex(418,-225){9}
    \Vertex(350,-304){12}
    \Line(375,-254)(350,-304)
    \Vertex(325,-254){9}
    \Line(350,-304)(325,-254)
  \end{picture}
}}\,}

\def\arbrefp{\,{\scalebox{0.15}{
\begin{picture}(8,156) (370,-147)
    \SetWidth{2}
    \SetColor{Black}
    \Line(374,-143)(374,-99)
    \Vertex(374,-96){9}
    \Vertex(375,-144){9}
    \Line(374,-92)(374,-48)
    \Vertex(374,-45){9}
    \Line(374,-42)(374,2)
    \Vertex(374,5){9}
    \Vertex(350,-194){12}
    \Line(375,-144)(350,-194)
    \Vertex(325,-144){9}
    \Line(325,-144)(350,-194)
    \Vertex(325,-96){9}
    \Line(325,-144)(325,-96)
  \end{picture}
}}\,}

\def\arbrefq{\,{\scalebox{0.15}{
  \begin{picture}(68,88) (329,-215)
    \SetWidth{2}
    \SetColor{Black}
    \Vertex(375,-212){12}
    \Line(376,-210)(400,-175)
    \Line(373,-211)(350,-174)
    \Vertex(350,-171){9}
    \Vertex(400,-173){9}
    \Line(347,-168)(330,-131)
    \Line(353,-168)(365,-133)
    \Vertex(330,-131){9}
    \Vertex(365,-131){9}
    \Line(400,-168)(400,-131)
    \Vertex(400,-131){9}
    \Vertex(365,-81){9}
    \Line(365,-131)(365,-81)
  \end{picture}
}}\,}

\def\arbrefqq{\,{\scalebox{0.15}{
  \begin{picture}(68,88) (329,-215)
    \SetWidth{2}
    \SetColor{Black}
    \Vertex(375,-212){12}
    \Line(376,-210)(400,-175)
    \Line(373,-211)(350,-174)
    \Vertex(350,-171){9}
    \Vertex(400,-173){9}
    \Line(398,-168)(385,-131)
    \Line(405,-168)(422,-133)
    \Vertex(422,-131){9}

    \Vertex(385,-131){9}
    \Line(350,-168)(350,-131)

    \Vertex(350,-131){9}
    \Vertex(422,-81){9}
    \Line(422,-81)(422,-131)
  \end{picture}
}}\,}

\def\arbrefr{\,{\scalebox{0.15}{ 
\begin{picture}(8,106) (370,-197)
    \SetWidth{2}
    \SetColor{Black}
    \Line(374,-193)(374,-149)
    \Vertex(374,-146){9}
    \Vertex(375,-194){9}
    \Line(374,-142)(374,-98)
    \Vertex(374,-95){9}
    \Line(425,-193)(425,-149)
    \Vertex(425,-146){9}
    \Vertex(425,-194){9}
    \Line(425,-142)(425,-98)
    \Vertex(425,-95){9}
    \Vertex(400,-244){12}
    \Line(425,-194)(400,-244)
    \Line(375,-194)(400,-244)
  \end{picture}
}}\,}

 \def\arbrefs{\,{\scalebox{0.15}{ 
  \begin{picture}(48,48) (349,-255)
    \SetWidth{2}
    \SetColor{Black}
    \Vertex(375,-252){9}
    \Line(375,-252)(395,-202)
    \Line(375,-252)(353,-202)
    \Vertex(353,-202){9}
    \Vertex(395,-202){9}
    \Line(430,-252)(430,-202)
    \Vertex(430,-202){9}
    \Vertex(430,-252){9}
    \Line(430,-202)(430,-152)
    \Vertex(430,-152){9}
    \Line(403,-302)(430,-252)
    \Line(403,-302)(375,-252)
    \Vertex(403,-302){12}
  \end{picture}
}}}

\def\arbrefss{\,{\scalebox{0.15}{ 
  \begin{picture}(48,48) (349,-255)
    \SetWidth{2}
    \SetColor{Black}
    \Vertex(430,-252){9}
    \Line(430,-252)(450,-202)
    \Line(430,-252)(410,-202)
    \Vertex(410,-202){9}
    \Vertex(450,-202){9}
    \Line(375,-252)(375,-202)
    \Vertex(375,-202){9}
    \Vertex(375,-252){9}
    \Line(375,-202)(375,-152)
    \Vertex(375,-152){9}
    \Line(403,-302)(430,-252)
    \Line(403,-302)(375,-252)
    \Vertex(403,-302){12}
  \end{picture}
}}}

\def\arbreft{\,{\scalebox{0.15}{ 
  \begin{picture}(48,48) (349,-255)
    \SetWidth{2}
    \SetColor{Black}
    \Vertex(445,-252){9}
    \Line(445,-252)(465,-202)
    \Line(445,-252)(425,-202)
    \Vertex(425,-202){9}
    \Vertex(465,-202){9}
  \Vertex(375,-252){9}
    \Line(375,-252)(395,-202)
    \Line(375,-252)(355,-202)
    \Vertex(395,-202){9}
    \Vertex(355,-202){9}
    \Line(410,-302)(445,-252)
    \Line(410,-302)(375,-252)
    \Vertex(410,-302){12}
  \end{picture}
}}}

\def\arbrefv{\,{\scalebox{0.15}{
 \begin{picture}(90,46) (330,-257)
    \SetWidth{2}
    \SetColor{Black}
    \Vertex(375,-254){12}
    \Line(376,-252)(395,-217)
    \Vertex(395,-215){9}
    \Line(374,-254)(335,-226)
    \Vertex(335,-226){9}
    \Line(375,-252)(356,-215)
    \Vertex(355,-215){9}
    \Line(374,-255)(417,-227)
    \Vertex(418,-225){9}
    \Line(335,-226)(335,-180)
    \Vertex(335,-180){9}
    \Line(335,-226)(335,-130)
    \Vertex(335,-130){9}
  \end{picture}
}}\,}
\def\arbrefvv{\,{\scalebox{0.15}{
 \begin{picture}(90,46) (330,-257)
    \SetWidth{2}
    \SetColor{Black}
    \Vertex(375,-254){12}
    \Line(376,-252)(395,-217)
    \Vertex(395,-215){9}
    \Line(374,-254)(335,-226)
    \Vertex(334,-224){9}
    \Line(375,-252)(356,-215)
    \Vertex(356,-215){9}
    \Line(374,-255)(417,-227)
    \Vertex(418,-225){9}
    \Line(356,-215)(356,-170)
    \Vertex(356,-170){9}
    \Line(356,-215)(356,-120)
    \Vertex(356,-120){9}
  \end{picture}
}}\,}
\def\arbrefvvv{\,{\scalebox{0.15}{
 \begin{picture}(90,46) (330,-257)
    \SetWidth{2}
    \SetColor{Black}
    \Vertex(375,-254){12}
    \Line(376,-252)(395,-217)
    \Vertex(395,-215){9}
    \Line(374,-254)(335,-226)
    \Vertex(335,-226){9}
    \Line(375,-252)(356,-215)
    \Vertex(356,-215){9}
    \Line(374,-255)(417,-227)
    \Vertex(418,-225){9}
    \Line(395,-215)(395,-170)
    \Vertex(395,-170){9}
    \Line(395,-215)(395,-120)
    \Vertex(395,-120){9}
  \end{picture}
}}\,}
\def\arbrefvvvv{\,{\scalebox{0.15}{
 \begin{picture}(90,46) (330,-257)
    \SetWidth{2}
    \SetColor{Black}
    \Vertex(375,-254){12}
    \Line(376,-252)(395,-217)
    \Vertex(395,-215){9}
    \Line(374,-254)(335,-226)
    \Vertex(335,-226){9}
    \Line(375,-252)(356,-215)
    \Vertex(355,-215){9}
    \Line(374,-255)(417,-227)
    \Vertex(418,-225){9}
    \Line(418,-225)(418,-180)
    \Vertex(418,-180){9}
    \Line(418,-225)(418,-130)
    \Vertex(418,-130){9}
  \end{picture}
}}\,}

\def\arbrefu{\,{\scalebox{0.15}{
 \begin{picture}(90,46) (330,-257)
    \SetWidth{2}
    \SetColor{Black}
    \Vertex(375,-254){12}
    \Line(376,-252)(395,-217)
    \Vertex(395,-215){9}
    \Line(374,-254)(335,-226)
    \Vertex(335,-226){9}
    \Line(375,-252)(356,-215)
    \Vertex(355,-215){9}
    \Line(374,-255)(417,-227)
    \Vertex(418,-225){9}
    \Line(335,-226)(353,-180)
    \Vertex(353,-180){9}
 \Line(335,-226)(317,-180)
    \Vertex(317,-180){9}
  \end{picture}
}}\,}
\def\arbrefuu{\,{\scalebox{0.15}{
 \begin{picture}(90,46) (330,-257)
    \SetWidth{2}
    \SetColor{Black}
    \Vertex(375,-254){12}
    \Line(376,-252)(395,-217)
    \Vertex(395,-215){9}
    \Line(374,-254)(335,-226)
    \Vertex(335,-226){9}
    \Line(375,-252)(356,-215)
    \Vertex(356,-215){9}
    \Line(374,-255)(417,-227)
    \Vertex(418,-225){9}
    \Line(356,-215)(374,-170)
    \Vertex(374,-170){9}
\Line(356,-220)(338,-170)
    \Vertex(338,-170){9}
  \end{picture}
}}\,}
\def\arbrefuuu{\,{\scalebox{0.15}{
 \begin{picture}(90,46) (330,-257)
    \SetWidth{2}
    \SetColor{Black}
    \Vertex(375,-254){12}
    \Line(376,-252)(395,-217)
    \Vertex(395,-215){9}
    \Line(374,-254)(335,-226)
    \Vertex(334,-224){9}
    \Line(375,-252)(356,-215)
    \Vertex(355,-215){9}
    \Line(374,-255)(417,-227)
    \Vertex(418,-225){9}
    \Line(395,-220)(413,-170)
    \Vertex(413,-170){9}
 \Line(395,-220)(377,-170)
    \Vertex(377,-170){9}
  \end{picture}
}}\,}
\def\arbrefuuuu{\,{\scalebox{0.15}{
 \begin{picture}(90,46) (330,-257)
    \SetWidth{2}
    \SetColor{Black}
    \Vertex(375,-254){12}
    \Line(376,-252)(395,-217)
    \Vertex(395,-215){9}
    \Line(374,-254)(335,-226)
    \Vertex(334,-224){9}
    \Line(375,-252)(356,-215)
    \Vertex(355,-215){9}
    \Line(374,-255)(417,-227)
    \Vertex(418,-225){9}
    \Line(418,-227)(400,-180)
    \Vertex(400,-180){9}
 \Line(418,-230)(436,-180)
    \Vertex(436,-180){9}
  \end{picture}
}}\,}

\def\arbrefw{\,{\scalebox{0.15}{
 \begin{picture}(90,46) (330,-257)
    \SetWidth{2}
    \SetColor{Black}
    \Vertex(375,-254){12}
    \Line(376,-252)(395,-217)
    \Vertex(395,-215){9}
    \Line(374,-254)(335,-226)
    \Vertex(335,-226){9}
    \Line(375,-252)(356,-215)
    \Vertex(356,-215){9}
    \Line(374,-255)(417,-227)
    \Vertex(418,-225){9}
    \Line(335,-225)(334,-180)
    \Vertex(334,-180){9}
\Line(356,-215)(356,-165)
    \Vertex(356,-165){9}
  \end{picture}
}}\,}

\def\arbrefww{\,{\scalebox{0.15}{
 \begin{picture}(90,46) (330,-257)
    \SetWidth{2}
    \SetColor{Black}
    \Vertex(375,-254){12}
    \Line(376,-252)(395,-217)
    \Vertex(395,-215){9}
    \Line(374,-254)(335,-226)
    \Vertex(335,-226){9}
    \Line(375,-252)(356,-215)
    \Vertex(355,-215){9}
    \Line(374,-255)(417,-227)
    \Vertex(418,-225){9}
    \Line(395,-215)(395,-170)
    \Vertex(395,-170){9}
    \Line(335,-226)(335,-176)
    \Vertex(335,-176){9}
  \end{picture}

}}\,}

\def\arbrefwww{\,{\scalebox{0.15}{
 \begin{picture}(90,46) (330,-257)
    \SetWidth{2}
    \SetColor{Black}
    \Vertex(375,-254){12}
    \Line(376,-252)(395,-217)
    \Vertex(395,-215){9}
    \Line(374,-254)(335,-226)
    \Vertex(334,-224){9}
    \Line(375,-252)(356,-215)
    \Vertex(356,-215){9}
    \Line(374,-255)(417,-227)
    \Vertex(418,-225){9}
    \Line(395,-215)(395,-170)
    \Vertex(395,-170){9}
    \Line(356,-215)(356,-170)
    \Vertex(356,-170){9}  
  \end{picture}
}}\,}

\def\arbrefwwww{\,{\scalebox{0.15}{
 \begin{picture}(90,46) (330,-257)
    \SetWidth{2}
    \SetColor{Black}
    \Vertex(375,-254){12}
    \Line(376,-252)(395,-217)
    \Vertex(395,-215){9}
    \Line(374,-254)(335,-226)
    \Vertex(335,-226){9}
    \Line(375,-252)(356,-215)
    \Vertex(355,-215){9}
    \Line(374,-255)(417,-227)
    \Vertex(418,-225){9}
    \Line(335,-226)(335,-176)
    \Vertex(335,-176){9}
    \Line(417,-227)(417,-177)
    \Vertex(417,-177){9}
  \end{picture}
}}\,}

\def\arbrefwwwww{\,{\scalebox{0.15}{
 \begin{picture}(90,46) (330,-257)
    \SetWidth{2}
    \SetColor{Black}
    \Vertex(375,-254){12}
    \Line(376,-252)(395,-217)
    \Vertex(395,-215){9}
    \Line(374,-254)(335,-226)
    \Vertex(334,-224){9}
    \Line(375,-252)(356,-215)
    \Vertex(356,-215){9}
    \Line(374,-255)(417,-227)
    \Vertex(418,-225){9}
    \Line(356,-215)(356,-165)
    \Vertex(356,-165){9}
    \Line(417,-227)(417,-177)
    \Vertex(417,-177){9}
  \end{picture}
}}\,}

\def\arbrefwwwwww{\,{\scalebox{0.15}{
 \begin{picture}(90,46) (330,-257)
    \SetWidth{2}
    \SetColor{Black}
    \Vertex(375,-254){12}
    \Line(376,-252)(395,-217)
    \Vertex(395,-215){9}
    \Line(374,-254)(335,-226)
    \Vertex(334,-224){9}
    \Line(375,-252)(356,-215)
    \Vertex(356,-215){9}
    \Line(374,-255)(417,-227)
    \Vertex(418,-225){9}
    \Line(395,-217)(395,-167)
    \Vertex(395,-167){9}
    \Line(417,-227)(419,-177)
    \Vertex(419,-177){9}
  \end{picture}
}}\,}

\def\arbrefx{\,{\scalebox{0.15}{
\begin{picture}(90,48) (330,-255)
    \SetWidth{2}
    \SetColor{Black}
    \Vertex(375,-252){12}
    \Line(375,-252)(390,-210)
    \Vertex(390,-210){9}
    \Line(375,-252)(330,-222)
    \Vertex(330,-222){9}
    \Line(375,-252)(350,-213)
    \Vertex(350,-213){9}
    \Line(375,-252)(410,-213)
    \Vertex(410,-213){9}
    \Line(375,-252)(370,-210)
    \Vertex(370,-210){9}
    \Line(375,-252)(430,-222)
    \Vertex(430,-222){9}

  \end{picture}}}\,}
\def\racineLabae{{\scalebox{0.3}{
\begin{picture}(12,12)(38,-38)
\SetWidth{0.5} \SetColor{Black} \Vertex(45,-28){5.66}
\Text(37,-20)[lb]{\Huge{\Black{$e$}}}
\end{picture}}}}

\def\racineLab1{{\scalebox{0.3}{
\begin{picture}(12,12)(38,-38)
\SetWidth{0.5} \SetColor{Black} \Vertex(45,-28){5.66}
\Text(37,-20)[lb]{\Huge{\Black{$1$}}}
\end{picture}}}}

\def\racineLabai{{\scalebox{0.3}{
\begin{picture}(12,12)(38,-38)
\SetWidth{0.5} \SetColor{Black} \Vertex(45,-28){5.66}
\Text(37,-20)[lb]{\Huge{\Black{$a_i$}}}
\end{picture}}}}

\begin{document}
%%%%%%%%%%%%%%%%%%%%%%%%%%%%%%%%%%%%%%%%%%%%%
%%%%%%%%%%%%%%%%%%%%%%%%%%%%%%%%%%%%%%%%%%%%%
%%%%%%%%%%%%%%%%%%%%%%%%%%%%%%%%%%%%%%%%%%%%%
\title[Pre-Lie Magnus Expansion]
      {On pre-Lie Magnus Expansion}

\author{Mahdi J. Hasan Al-Kaabi}

\address {Mathematics Department, College of Science, Al-Mustansiriya University, Palestine Street, P.O.Box 14022, Baghdad, IRAQ. E-mail address:Mahdi.Alkaabi@math.univ-bpclermont.fr}

%%%%%%%%%%%%%%%%%%%%%%%%%%%%%%%%%%%%%%%%%%%%%%%%%%%%%%%%%%%%%%%%%%%
\date{}
%%%%%%%%%%%%%%%%%%%%%%%%%%%%%%%%%%%%%%%%%%%%%%%%%%%%%%%%%%%%%%%%%%%
\begin{abstract}
In this paper, we study the classical and pre-Lie Magnus expansions, discussing how we can find a recursion for the pre-Lie case which already incorporates the pre-Lie identity. We give a combinatorial vision of a numerical method proposed by S. Blanes, F. Casas, and J. Ros in \cite{BFJ2000}, on a writing of the classical Magnus expansion in a free Lie algebra, using a pre-Lie structure.   
\end{abstract}

\maketitle
\tableofcontents

%%%%%%%%%%%%%%%%%%%%%%%%%%%%%%%%%%%%%%%%%%%%%%%%%%%%%%%%%%%%%%%%%%%
%%%%%%%%%%%%%%%%%%%%%%%%%%%%%%%%%%%%%%%%%%%%%%%%%%%%%%%%%%%%%%%%%%%
\section{Introduction}
Wilhelm Magnus (1907-1990) is a topologist, an algebraist, an authority on differential equations and on special functions and a mathematical physicist. One of his long-lasting constructions is a tool to solve classical linear differential equations for linear operators, now named Magnus expansion \cite{WM54}, which has found applications in numerous areas, in particular in quantum chemistry and theoretical physics.\\

Magnus expansion is a formal expansion of the logarithm of the solution of the following linear differential equation:

\begin{equation}
\dot{y}(t):= \frac{d}{dt} y(t) = a(t) y(t), \;y(0)=1.
\end{equation}

Many works have been raised to write the classical Magnus expansion in terms of algebro-combinatorial structures: Rota-Baxter algebras, dendriform algebras, pre-Lie algebras and others, see for example \cite{KM09, K.M09, CP13} for more details about these works. Particularly, we study a generalization of the latter called pre-Lie Magnus expansion \cite{AR81}, and we give a brief survey about this expansion in this paper. An approach has been developed to encode the terms of the classical and pre-Lie Magnus expansions respectively, by A. Iserles with S. P. No\!\!/\;\!rsett using planar binary trees \cite{IN99}, and by K. Ebrahimi-Fard with D. Manchon using planar rooted trees \cite{K.M09} respectively.\\

F. Chapoton and F. Patras introduced a concrete formula in \cite{CP13} using the Grossman-Larson algebra. We study this formula briefly in Sections \ref{P-LME} and \ref{AMERT}, and we compare its terms with another pre-Lie Magnus expansion terms obtained by K. Ebrahimi-Fard and D. Manchon in \cite{K.M09}. We observe that this formula can be considered as optimal up to degree seven, with respect to the number of terms in the pre-Lie Magnus expansion. The question, raised by K. Ebrahimi-Fard, of writing an optimal (i.e. with a minimal number of terms) pre-Lie Magnus expansion at any order remains open.\\

In Section \ref{section IIII}, we look at the pre-Lie Magnus expansion in the free Lie algebra $\Cal{L}(E)$. The weighted anti-symmetry relations lead to a further reduction of the number of terms. The particular case of one single generator in each degree is closely related to the work of S. Blanes, F. Casas and J. Ros \cite{BFJ2000}. We give a combinatorial interpretation of this work, using the monomial basis of free Lie algebra $\Cal{L}(E)$\label{CL} described in \cite{AM20}.

\section{Classical Magnus expansion}\label{CME} 
W. Magnus provides an exponential representation of the solution of the well-known classical initial value problem:
\begin{equation}\label{diff.eq}
\dot{Y}(t):= \frac{d}{dt} Y(t) = A(t) Y(t), \hbox{ with initial condition } Y(0)=1,
\end{equation}
where $Y(t),\,A(t)$ are linear operators depending on a real variable $t$, and $1$ is the identity operator. Magnus considers the problem \eqref{diff.eq} in a non-commutative context. The problem, according to Magnus' point of view, is to define an operator $\Omega(t)$, depending on $A$, with $\Omega(0)=0$ such that :
\begin{equation}
Y(t) = exp\big(\int\limits_{0}^{t}{\dot{\Omega}(s) ds}\big) = \sum\limits_{n\geq 0}^{}{\frac{\Omega(t)^n}{n!}}.
\end{equation}
 He obtains a differential equation leading to the recursively defined expansion named after him:
\begin{equation}\label{magnus formula}
\Omega(t) = \int\limits_{0}^{t}{\dot{\Omega}(s) ds }= \int\limits_{0}^{t}{A(s) ds} + \int\limits_{0}^{t} \sum_{n > 0} \frac{B_n}{n!} ad^{(n)}_{\int\limits_{0}^{s}{\dot{\Omega}(u) du }}{[A(s)]ds}, 
\end{equation} 
where $B_n$ are the Bernoulli numbers defined by:
$$\sum\limits_{m=0}^{\infty}{\frac{B_m}{m!}z^m} = \frac{z}{e^z-1} = 1 - \frac{1}{2}z + \frac{1}{12} z^2 - \frac{1}{720} z^4 + \cdots,$$
and $ad_{\Omega}$ is a shorthand for an iterated commutator:
$$ad^{0}_{\Omega} A = A, \,\, ad^{1}_{\Omega} A= [\Omega, A] = \Omega A - A \Omega, \,\, ad^{2}_{\Omega} A = [\Omega, [\Omega, A]],$$
and, in general, $ad_{\Omega}^{m+1} A = [\Omega, ad_{\Omega}^{m} A]$ \cite{WM54}.  Taking into account the numerical values of the first few Bernoulli numbers, the formula in \eqref{magnus formula} can be written: 
$$ \dot{\Omega}(t) = A(t) - \frac{1}{2} [ \Omega, A(t)] + \frac{1}{4} \big[[\Omega, A(t)], \Omega \big] + \frac{1}{12} \big[\Omega, [\Omega, A(t)]\big] + \cdots,$$
where $\dot{\Omega}(t) := \Omega'(t) = \frac{d}{dt} \Omega(t)$. Also, we can write the expansion in \eqref{magnus formula} as:

\begin{equation}
\Omega(t) = \sum_{n > 1}{\Omega}_n(t),
\end{equation}
where $\Omega_1(t) = \int\limits_{0}^{t}{A(s)\,ds}$, and in general:
\begin{equation}\label{generating formula}
\Omega_n(t) = \sum_{j=1}^{n-1}{\frac{B_j}{j!} \int\limits_{0}^{t}{S_n^{(j)}(s)\,ds}}, \hbox{ for } n \geq 2,
\end{equation}
where $S_n^{(1)} = [\Omega_{n-1}, A],\;S_n^{(n-1)}= ad_{\Omega_1}^{\,n-1}\big(A\big)$, and:
$$S_n^{(j)} = \sum_{m=1}^{n-j}{\big[\Omega_m, S_{n-m}^{(j-1)}\big]}, \hbox{ for } 2 \leq j \leq n-1.$$ 
The formula \eqref{generating formula} can be found in \cite{WM54}, \cite{BCR09}.

\section{Pre-Lie Magnus expansion}\label{P-LME}
In this section, we study an important generalization of the Magnus expansion in the pre-Lie setting: let $(\Cal{P\!L}, \rhd)$ be a pre-Lie algebra defined over a field $K$\label{K11}. The linear transformations $L_A$, for $A \in \Cal{P\!L}$, defined by $L_A: \Cal{P\!L} \rightarrow \Cal{P\!L}$, such that $L_{A}(B):= A \rhd B, \hbox{ for all } B \in \Cal{P\!L}$. Define $\dot{\Omega}:= \dot{\Omega}(\lambda A)$, for $A \in \Cal{P\!L}$, to be a formal power series in $\lambda\Cal{P\!L}[[\lambda]]$. Now, the classical Magnus expansion, described in \eqref{magnus formula}, can be rewritten as:

\begin{equation}\label{pre-Lie formula A}
\dot{\Omega}(\lambda A)(x) = \frac{L_{\rhd}[\dot{\Omega}]}{exp(L_{\rhd}[\dot{\Omega}]) - 1} (\lambda A)(x) = \sum_{m \geq 0}{\frac{B_m}{m!}L_{\rhd}[\dot{\Omega}]^{m}(\lambda A)(x)},
\end{equation} 
where $ L_{\rhd}[\dot{\Omega}]\big(\lambda A\big)(x) =  \big( \dot{\Omega} \rhd \lambda A \big)(x) = [ \int\limits_{0}^{x}{\dot{\Omega}(s) ds}, \lambda A(x)]$, $B_m$ are Bernoulli numbers, this formula is called pre-Lie Magnus expansion \cite{AR81}, \cite{KM09}. 

\begin{lem}
Let $A, B$ be linear operators depending on a real variable $x$, then the product: 
\begin{equation}
(A \rhd B)(x) := \Big[\!\int\limits_{0}^{x}{\!\!A(s) ds}, B(x)\Big]\,,
\end{equation}
verifies the pre-Lie identity, where $ [A(x), B(x)] = (A \cdot B - B \cdot A)(x)$.
\end{lem}

\begin{proof}
Let $A, B, C$ be linear operators depending on a real variable $x$. Set $I(A)(x) := \int\limits_{0}^{x}{\!A(s) ds}$, then we have:

\begin{equation}\label{prop. of Int.}
I(A) \cdot I(B) = I \big( I(A) \cdot B + A \cdot I(B) \big),
\end{equation}

In other words, $I$ is a weight zero Rota-Baxter operator \footnote{For more details about Rota-Baxter operator, Rota-Baxter algebras see \cite[Paragraph 5.2]{KM09} and the references therein.}. Hence,

\begin{align*}
\hspace{-15pt}\Big(\big(A \rhd B\big) \rhd C\Big)(x) - \Big(A \rhd \big(B \rhd C\big)\Big)(x) & = \big[I\big([I(A), B]\big)(x), C(x)\big] - \big[[I(A)(x), [I(B)(x), C(x)]\big]&\\
& = \big[I\big([I(A), B]\big)(x), C(x)\big] - \Big( \big[[I(A)(x), I(B)(x)], C(x)\big]&\\
& \;\;\;\; + \big[I(B)(x), [I(A)(x), C(x)]\big] \Big)\;\;\;, \hbox{ (by the Jacobi identity), }&\\
& = \big[I\big(I(A) \cdot B - B \cdot I(A) \big)(x), C(x)\big] - \big[\big(I(A) \cdot I(B)- I(B) \cdot I(A) \big)(x), C(x)\big]&\\
& \;\;\;\; - \big[I(B)(x), [I(A)(x), C(x)]\big] &\\
& = \big[I\big(I(A) \cdot B \big)(x) - \big(I(A) \cdot I(B)\big)(x) + \big(I(B) \cdot I(A) \big)(x) - I\big(B \cdot I(A) \big)(x), C(x)\big] &\\
& \;\;\;\; - \big[I(B)(x), [I(A)(x), C(x)]\big] &\\
& = \big[I\big(I(B) \cdot A - A \cdot I(B)\big)(x), C(x)\big] - \big[I(B)(x), [I(A)(x), C(x)]\big]\;, \hbox{ (by \eqref{prop. of Int.}), }&\\
& = \big[I\big([I(B), A]\big)(x), C(x)\big] - \big[[I(B)(x), [I(A)(x), C(x)]\big]&\\
& = \Big(\big(B \rhd A\big) \rhd C\Big)(x) - \Big(B \rhd \big(A \rhd C\big)\Big)(x).&
\end{align*}
This proves the Lemma.
\end{proof}

\noindent Also, the formula \eqref{pre-Lie formula A} can be represented as\label{om1}:
\begin{equation}
\dot{\Omega}(\lambda A) = \sum_{ n > 0} { \dot{\Omega}_n(\lambda A)}, 
\end{equation}
where $\dot{\Omega}_1(\lambda A) = \lambda A$, and in general: 
\begin{equation}
\dot{\Omega}_n(\lambda A) = \sum\limits_{j=1}^{n-1}{\frac{B_j}{j!}\sum_{{k_1 + \cdots + k_j =\,n-1}\atop {k_1 \geq 1,\,\ldots,\,k_j \geq 1}}{\!\!\!\!L_{\rhd}[\dot{\Omega}_{k_1}]\big(L_{\rhd}[\dot{\Omega}_{k_2}] \big( \cdots ( L_{\rhd}[\dot{\Omega}_{k_j}](\lambda A))\cdots)\big)}}, \hbox{ for } n\geq 2.
\end{equation}
Here, we give few first terms of the pre-Lie Magnus expansion described above: 

\begin{align}\label{pre-Lie formula B}
\!\!\!\!\!\!\!\!\!\!\!\!\!\!\!\!\!\!\!\!\!\!\!\!\!\!\!\!\!\!\!\!\!\!\!\!\!\!\!\!\dot{\Omega}(\lambda A) = \,\,\lambda A - \lambda^2\frac{1}{2} (A \rhd A) + \lambda^3 \big(\frac{1}{4} (A \rhd A) \rhd A+ \frac{1}{12} A \rhd (A \rhd A)\big) &&
\end{align} 
\begin{align*}
&& \;\;\;\;- \lambda^4 \left( \frac{1}{8} ((A \rhd A) \rhd A) \rhd A + \frac{1}{24} ( A \rhd (A \rhd A)) \rhd A + \frac{1}{24} \big( A \rhd ((A \rhd A) \rhd A) + (A \rhd A ) \rhd (A \rhd A) \big) \right) + \Cal{O}(\lambda^5)  
\end{align*} 
 
There are many ways of writing the Magnus expansion, for pre-Lie and classical formulas, in various settings using Baker-Campbell-Hausdorff series, dendriform algebras, Rota-Baxter algebras, Solomon Idempotents and others, for more details about these works see \cite{AR81}, \cite{KM09}, \cite{CP13} and the references therein.\\

Using the pre-Lie identity, the pre- Lie Magnus expansion terms can be reduced: for the terms at third order, $\dot{\Omega}_3(\lambda A)$, no further reduction of terms is possible.  At fourth order, two terms can be reduced as follows:
\begin{equation}\label{fourth order}
\dot{\Omega}_4(\lambda A) = \lambda^4 \left( \frac{1}{8} \big((A \rhd A) \rhd A\big) \rhd A + \frac{1}{24} \big(( A \rhd (A \rhd A)\big) \rhd A + A \rhd \big((A \rhd A) \rhd A\big) + (A \rhd A ) \rhd (A \rhd A) \right)
\end{equation}
and, by pre-Lie identity, we have:
$$ (A \rhd A ) \rhd (A \rhd A) = \big( (A \rhd A) \rhd A\big) \rhd A - \big( A \rhd ( A \rhd A )\big) \rhd A + A \rhd \big((A \rhd A) \rhd A\big), $$
then \eqref{fourth order} equals:
$$ \lambda^4 \Big(\frac{1}{6} \big((A \rhd A) \rhd A\big) \rhd A + \frac{1}{12} x \rhd \big((A \rhd A) \rhd A\big)\Big).$$
At fifth order, $\dot{\Omega}_5(\lambda A)$, three terms out of ten can be removed \cite{KM09}. For more details about this reduction of pre-Lie Magnus expansion terms, see the next sections.\\

A beautiful way of writing the pre-Lie Magnus expansion is proposed by F. Chapoton and F. Patras in their joint work \cite{CP13}. We review here a part of their work corresponding to pre-Lie Magnus element, as follows: let $\big(\Cal{P\!L}(a), \rhd\big)$ be the free pre-Lie algebra with one generator $a$, and $\widehat{\Cal{P\!L}}(a)$ be its completion\footnote{For further details about the completed pre-Lie algebra see the references \cite{NJ80, AR81, D.M}.}. The Magnus element in $\widehat{\Cal{P\!L}}(a)$ is the (necessarily unique) solution $\dot{\Omega}$ to the equation\label{not-S}:

\begin{equation}
\dot{\Omega} = \Big( \frac{\dot{\Omega}}{exp(\dot{\Omega}) - 1 }\Big) \rhd a.
\end{equation}

The exponential series $exp(a) := \sum\limits_{n \geq 0}^{}{\frac{a ^n}{n !}}$ belongs to $\widehat{S\big(\Cal{P\!L}\big)}$, the completion of the symmetric algebra over $\Cal{P\!L}(a)$, endowed with its usual commutative algebra structure. We give in following an important result obtained by F. Chapoton and F. Patras in \cite{CP13}.     

\begin{thm}
The Magnus element $\dot{\Omega}(a)$ in $\widehat{\Cal{P\!L}}(a)$ can be written:
\begin{equation}\label{logarithm for.}
\dot{\Omega}(a) = log^{\ast}\big(exp(a)\big),
\end{equation}
where $\ast$ is the Grossman-Larson product\footnote{ Grossman-Larson algebra is defined in the next section, Paragraph \ref{section of calcualtions}.}. The notation $log^{\ast}$ means that the logarithm is computed with respect to the product $\ast$.  
\end{thm}
\begin{proof}
See \cite[Theorem 4]{CP13}.
\end{proof}
     
\section{An approach for Magnus expansion terms using rooted trees}\label{AMERT}
A. Iserles and S. P. No\!\!/\;\!rsett have developed an alternative approach, using planar binary rooted trees to encode the classical Magnus expansion terms \cite{IN99}. K. Ebrahimi-Fard and D. Manchon, in their joint work \cite{K.M09}, used planar rooted trees to represent the pre-Lie Magnus expansion. This encoding of expansion terms, using planar binary rooted trees, is defined as:
$$x  \rightsquigarrow \treel,\;\; x \rhd x  \rightsquigarrow \treesmall\;.$$

Hence, the pre-Lie Magnus expansion, described in \eqref{pre-Lie formula B}, can be denoted in the shorthand as:

\begin{equation}
\dot{\Omega}(\treel) = \treel - \frac{1}{2}\;\treesmall + \frac{1}{4}\;\treeB + \frac{1}{12}\;\treeA - \left( \frac{1}{8}\;\treeE + \frac{1}{24}\;\left( \treeF + \treeG + \treeD\;\right)\right) + \cdots
\end{equation}
and the reduction in expansion terms at the fourth order can be described as:
$$\dot{\Omega}_4(\treel) = -\,\frac{1}{6}\;\treeE - \frac{1}{12}\;\treeG\;,$$
thanks to the pre-Lie identity:
$$ \treeD - \treeE = \treeG - \treeF\;.$$

The approach proposed by K. Ebrahimi-Fard and D. Manchon is more in the line of non-commutative Butcher series\footnote{For more details about Butcher series see \cite[Section 4.1]{CB00}.}. In following, we shall review the joint work of K. Ebrahimi-Fard and D. Manchon, published in \cite{K.M09}, on finding an explicit formula, in planar rooted tree version, for pre-Lie Magnus expansion. Let $\sigma = B_+(\sigma_1 \cdots \sigma_k)$\label{op1} be any (undecorated) planar rooted tree, denote $f(v)$, for $v \in V(\sigma)$, by the number of outgoing edges, i.e. the fertility of the vertex $v$ of $\sigma$. The degree $|\sigma|$ of a tree here is given by the number of its vertices. Define the linear map $\gamma: \Cal{T}_{\!\!pl} \rightarrow K$\label{K12} as:
\begin{equation}\label{definition of alph.}
\gamma(\sigma):= \frac{B_k}{k!} \prod\limits_{i=1}^{k}{\gamma(\sigma_i)} = \prod_{v\,\in V(\sigma)}{\frac{B_{f(v)}}{f(v)!}}\,, 
\end{equation}
where $B_k$ are Bernoulli numbers.
\begin{lem}
For any planar rooted tree $\tau$, such that there exists $v \in V(\tau)$ of fertility $2n +1, n > 0$, we have $\gamma(\tau)=0$. 
\end{lem}

\begin{proof}
It is immediate from the definition of $\gamma$ in \eqref{definition of alph.}, and the fact that $B_{2n+1} = 0$, for all $n > 0$.
\end{proof} 
Define a subspace $\Cal{T}^{e1}_{\!\!pl}$\label{e11} of all planar rooted trees excluding trees with at least one vertex of fertility $2n + 1$, with $n > 0$. The tree functional $F$ is defined recursively by:
\begin{equation}\label{def. of F}
F[\racine](x) = x, \hbox{ and } F[\tau](x):= r_{\rhd}^{(k+1)}\big(F[\tau_1](x), \ldots, F[\tau_k](x), x\big),
\end{equation} 
where $ \tau= B_+(\tau_1 \cdots \tau_k)$, and
$$ r_{\rhd}^{(k+1)}\big(F[\tau_1](x), \ldots, F[\tau_k](x), x\big):= F[\tau_1](x) \rhd ( F[\tau_2](x) \rhd (\cdots \rhd (F[\tau_k](x) \rhd x )\cdots)).$$
\begin{thm}\label{NRT}
The pre-Lie Magnus expansion can be written:
\begin{equation}\label{KM formula}
\dot{\Omega}(x) = \sum_{ \tau \in T^{e1}_{\!\!pl}}{\gamma(\tau) F[\tau](x)}.
\end{equation}
\end{thm} 

\begin{proof}
See \cite[Theorem 20]{K.M09}.
\end{proof} 

For $ n \geq 1$, the numbers of trees in $T^{e1,\,n}_{\!pl}$, the subset of all planar rooted trees with "$1 \hbox{ or even fertility }$" of degree $n$, is given by the sequence "A049130" in \cite{S}. Here, we give few of first terms of this sequence: $1, 1, 2, 4, 10, 26, 73, 211, 630, \ldots$. \\

We give here some examples of the formula of pre-Lie Magnus expansion described in \eqref{KM formula}, as follows:

\begin{align*}
\dot{\Omega}(x) &= \gamma(\racine) F[\racine](x) + \gamma(\arbrea)F[\arbrea](x) + \gamma(\arbreba)F[\arbreba](x) + \gamma(\arbrebb)F[\arbrebb](x) + \Cal{O}(4)&\\ 
& = x - \frac{1}{2} r^{(2)}_{\rhd}\big(F[\racine](x), x\big) + \frac{1}{4} r^{(2)}_{\rhd}\big(F[\arbrea](x), x\big) + \frac{1}{12} r^{(3)}_{\rhd}\big(F[\racine](x), F[\racine](x), x\big) + \Cal{O}(4).&
\end{align*}
At order four, we have:

\begin{align*}
\dot{\Omega}_4(x) &= \gamma(\arbreca) F[\arbreca](x) + \gamma(\arbrecb) F[\arbrecb](x) + \gamma(\arbrecc) F[\arbrecc](x) + \gamma(\arbreccc) F[\arbreccc](x)&\\
&= - \left( \frac{1}{8} r^{(2)}_{\rhd}\big(F[\arbreba](x), x\big) + \frac{1}{24} r^{(2)}_{\rhd}\big(F[\arbrebb](x), x\big) + \frac{1}{24} \big( r^{(3)}_{\rhd}\big(F[\arbrea](x), F[\racine](x), x\big) + r^{(3)}_{\rhd}\big(F[\racine](x), F[\arbrea](x), x\big) \big) \right),   
\end{align*}
but, thanks to pre-Lie identity, we have:
$$ r^{(2)}_{\rhd}\big(F[\arbreba](x), x\big) - r^{(3)}_{\rhd}\big(F[\arbrea](x), F[\racine](x), x\big) = r^{(2)}_{\rhd}\big(F[\arbrebb](x), x\big) - r^{(3)}_{\rhd}\big(F[\racine](x), F[\arbrea](x), x\big) \big),$$
then the formula $\dot{\Omega}_4(x)$ can be reduced into two terms, as follows:
\begin{align}\label{order four}
\dot{\Omega}_4(x) = & - \frac{1}{6} r^{(2)}_{\rhd}\big(F[\arbreba](x), x\big) - \frac{1}{12} r^{(3)}_{\rhd}\big(F[\racine](x), F[\arbrea](x), x\big)&    
\end{align} 
\begin{align*}
\!\!\!\!\!\!\!\!& = - \frac{1}{6} F[\arbreca](x)  - \frac{1}{12} F[\arbreccc](x).&
\end{align*}

Eight trees out of ten appear in the pre-Lie Magnus expansion at order five, due to the recursive nature of this expansion, which are\label{om2}:
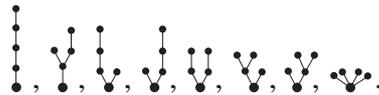
$$ \arbreda\;, \arbredcc , \arbrede, \arbredee, \arbredf, \arbredz,\!\!\arbredzz , \arbredh. $$
Hence,

\begin{align*}
\dot{\Omega}_5(x) = & \gamma(\arbreda) F[\arbreda](x) + \gamma(\arbredcc) F[\arbredcc](x) + \gamma(\arbrede) F[\arbrede](x) + \gamma(\arbredee) F[\arbredee](x) + \gamma(\arbredf) F[\arbredf](x) & \\
& + \gamma(\arbredz) F[\arbredz](x) + \gamma(\!\arbredzz) F[\!\arbredzz](x) + \gamma(\arbredh) F[\arbredh](x).& 
\end{align*}
Using the pre-Lie identity as:
$$ F[\arbredcc](x) - F[\arbredee](x) = F[\arbreda](x) - F[\arbrede](x),$$
we obtain a reduced formula of pre-Lie Magnus expansion at order five, with seven terms described as:

\begin{align}\label{order five}
\dot{\Omega}_5(x) = & \frac{5}{48} F[\arbreda](x) + \frac{1}{48} F[\arbredcc](x) + \frac{1}{24} F[\arbredee](x) + \frac{1}{48} F[\arbredf](x) &
\end{align}
\begin{align*}
\;\;\;\;\;\;\;& + \frac{1}{144} \left( F[\arbredz](x) + F[\!\arbredzz](x) \right) - \frac{1}{120} F[\arbredh](x).& 
\end{align*} 

The reduced formulas at orders four and five, described in \eqref{order four}, \eqref{order five} respectively, are considered as best (or optimal) formulas for the pre-Lie Magnus expansion at these orders. 

\subsection{Some calculations in pre-Lie Magnus expansion}\label{section of calcualtions}\mbox{}

Let us consider the free pre-Lie algebra $\Cal{P\!L} = \Cal{T}$\label{cnpnd1} with one generator $\racine$, together with the pre-Lie grafting $\to$\label{to4}. Then, we can represent pre-Lie Magnus expansion in terms of rooted trees as in the following. We need first to introduce the following result.

\begin{lem}\label{rel. between F and Psi}
For any planar rooted tree $\tau$, we have:
$$ F[\tau](\racine) = \overline{\Psi}(\tau),$$
where $F$ is the function described in \eqref{def. of F}, and $\overline{\Psi}$ is defined in our work \cite[Subsection 2.2]{AM2014}.
\end{lem} 

\begin{proof}
Let $\tau$ be any planar rooted tree with $k$ branches, then it can be written in a unique way as $\tau = B_+(\tau_1 \ldots \tau_k)$. Using the induction hypothesis on the number $k$ of branches, we have:

$$ F[\racine](\racine) = \overline{\Psi}(\racine) = \racine\,.$$ 

\noindent Suppose that the hypothesis of this Lemma is true for all planar rooted trees $\tau'$ with $n-1$ branches, for all $n \leq k$, i.e. $F[\tau'](\racine) = \overline{\Psi}(\tau')$, hence:

\begin{align*} 
F[\tau](\racine) &= r_{\to}^{(k+1)}\big(F[\tau_1](\racine), \ldots, F[\tau_k](\racine), \racine\big) \hskip 8mm \hbox{ (from definition of $F$ in \eqref{def. of F})}&\\
& = F[\tau_1](\racine) \to \Big( F[\tau_2](\racine) \to \big( \cdots \to \big( F[\tau_k](\racine) \to \racine \big) \cdots\big)\Big)&\\
& = \overline{\Psi}(\tau_1) \to \Big( \overline{\Psi}(\tau_2) \to \big( \cdots \to \big( \overline{\Psi}(\tau_k) \to \racine \big) \cdots\big)\Big)\hskip 8mm \hbox{ (by the hypothesis above)}&\\
& = \overline{\Psi}\big(\tau_1\lbutcher (\tau_2 \lbutcher ( \cdots \lbutcher (\tau_k \lbutcher \racine ) \cdots))\big)\hskip 8mm \hbox{ (from definition of $\overline{\Psi}$)}&\\
& = \overline{\Psi}(\tau)\hskip 8mm \hbox{ (since $\tau = B_+(\tau_1 \ldots \tau_k) = \tau_1\lbutcher (\tau_2 \lbutcher ( \cdots \lbutcher (\tau_k \lbutcher \racine ) \cdots))$ ) }.&
\end{align*}
This proves the Lemma.
\end{proof} 

\begin{prop}
The pre-Lie Magnus expansion can be written as:
\begin{equation}\label{formula D}
\dot{\Omega}(\racine) = \sum_{{\tau \in T^{e1}_{_{\!\!pl}}}\atop{s\,\in T}}{\gamma(\tau)\,\alpha(s, \tau) s,}
\end{equation}
where $\alpha(s, \tau)$\label{e12} are the coefficients described in \cite[Theorem 4]{AM2014}, and $\gamma$ is the map defined above in \eqref{definition of alph.}.
\end{prop}
\begin{proof}
 Immediate from Theorem \ref{NRT} and Lemma \ref{rel. between F and Psi}, and using the formula:
$$\overline{\Psi}(\tau) = \sum_{s\,\in T}{\alpha(s, \tau) s},$$ 
that is introduced by \cite[Theorem 4]{AM2014}. 
\end{proof}
Now for any $\tau \in T^{e1}_{\!pl}$, let  $ e_{\tau} := \overline{\Psi}(\tau)$. The planar rooted tree $\tau$ is uniquely written as a monomial expression $m(\racine, \lbutcher)$ involving  the root and the left Butcher product. Then $\overline{\Psi}(\tau)$ is $m(\racine, \to)$\label{to5}, i.e. the same monomial expression where the left Butcher product is replaced by the pre-Lie grafting of (non-planar) rooted trees. Here, we display optimal (with respect to the number of terms) formulas of pre-Lie Magnus expansion up to order seven:\\
\begin{align*}
\dot{\Omega}_{1}(\!\racine)&= \racine\\ 
\dot{\Omega}_{2}(\!\racine)&= B_{1}\ e_{_{\arbrea}}\\ 
\dot{\Omega}_{3}(\!\racine)&= B^2_{1}\ e_{_{\arbreba}} + \frac{B_2}{2!}\ e_{_{\arbrebb}}\\ 
\dot{\Omega}_{4}(\!\racine)&= \frac {B_1}{3}\ e_{_{\arbreca}} + B_1B_2\ e_{_{\!\!\!\!\!\!\arbreccc}}\\
\dot{\Omega}_{5}(\!\racine)&= -B_1\frac{B_2}{2!} \frac{5}{2}\ e_{_{\arbreda}}-B_1\frac{B_2}{2!} \frac{1}{2}\ e_{_{\!\!\!\!\!\!\arbredcc}} + B_{1}^{2}B_2\ e_{_{\!\!\!\!\!\!\arbredee}} + B_{1}^{2}\frac{B_2}{2!}\ e_{_{\arbredf}} + \frac{B_{2}^{2}}{2!2!}\ (e_{_{\arbredz}}+e_{_{\!\!\!\!\!\!\arbredzz}}\ ) + \frac{B_4}{4!}\ e_{_{\arbredh}}\\\\ 
\dot{\Omega}_{6}(\!\racine)&= -\frac{11}{144}\ e_{_{\arbreea}}-\frac{5}{288}\ e_{_{\!\!\!\!\!\!\!\arbreeb}} - \frac{1}{288}\ (e_{_{\arbreec}}+e_{_{\!\!\!\!\!\!\arbreecc}}\ ) + B_1\frac{B_4}{4!}\ e_{_{\arbreed}}-\frac{1}{36}\ e_{_{_{_{_{_{_{_{_{\!\!\!\!\!\!\!\arbreee}}}}}}}}}- \frac{1}{144} (e_{_{_{_{_{_{_{_{\!\!\!\!\!\!\arbreef}}}}}}}} + e_{_{_{_{_{_{_{_{\!\!\!\!\!\!\!\!\!\!\!\!\arbreeff}}}}}}}}\ )- \frac{1}{48}\ e_{_{_{_{_{_{_{_{_{\!\!\!\!\!\!\arbreeg}}}}}}}}} \\
& \ \ \ \ - \frac{1}{288} ( e_{_{\arbreeh}} + e_{_{\!\!\!\arbreehh}}\ ) + B_1\frac{B_4}{4!}\ (e_{_{_{_{_{_{_{\!\arbreei}}}}}}} + e_{_{_{_{_{_{_{_{\!\!\!\arbreeii}}}}}}}} + e_{_{_{_{_{_{_{_{\!\!\!\!\!\!\arbreeiii}}}}}}}} + e_{_{_{_{_{_{_{\!\!\!\!\!\!\!\!\!\arbreeiiii}}}}}}})&\\\\
\dot{\Omega}_{7}(\!\racine)&= \frac{31}{576}\ e_{_{\arbrefa}} + \frac{1}{576}\ e_{_{\!\!\!\!\!\!\!\arbrefb}} + \frac{1}{576}\ (e_{_{\arbrefc}}+e_{_{\!\!\!\!\!\!\arbrefcc}}\ ) + B^{2}_{1}\frac{B_4}{4!}\ e_{_{\arbrefd}} + \frac{7}{576}\ e_{_{_{_{_{_{_{_{_{\!\!\!\!\!\!\!\arbrefe}}}}}}}}} + \frac{1}{288} (e_{_{_{_{_{_{_{_{\!\!\!\!\!\!\arbreff}}}}}}}} + e_{_{_{_{_{_{_{_{\!\!\!\!\!\!\!\!\!\!\!\!\arbrefff}}}}}}}}\ ) + \frac{1}{288}\ e_{_{_{_{_{_{_{_{_{\!\!\!\!\!\!\arbrefg}}}}}}}}} \\\\
& \ \ \ \ + \frac{1}{576} ( e_{_{\arbrefh}} + e_{_{\!\!\!\arbrefhh}}\ ) + B^{2}_{1}\frac{B_4}{4!}\ (e_{_{_{_{_{_{_{\!\arbrefi}}}}}}} + e_{_{_{_{_{_{_{_{\!\!\!\arbrefii}}}}}}}} + e_{_{_{_{_{_{_{_{\!\!\!\!\!\!\arbrefiii}}}}}}}} + e_{_{_{_{_{_{_{\!\!\!\!\!\!\!\!\!\arbrefiiii}}}}}}}) + \frac{5}{288} \ e_{_{\arbrefj}} + \frac{1}{576}\ ( e_{_{\!\!\!\!\!\!\arbrefk}} + e_{_{\!\!\!\!\!\!\arbrefkk}} ) + \frac{1}{288}\ ( e_{_{\!\!\!\!\!\!\arbrefl}} + e_{_{\!\!\!\!\!\!\arbrefll}} )\\\\
&\ \ \ \  + \frac{1}{576}\ ( e_{_{\arbrefm}} + e_{_{\arbrefmm}} ) + \frac{1}{1728}\ ( e_{_{\arbrefn}} + e_{_{\arbrefnn}} + e_{_{\!\!\!\!\!\!\arbrefnnn}} + e_{_{\!\!\!\!\!\!\arbrefnnnn}}\ ) + \frac{B_2}{2!}\frac{B_4}{4!}\ ( e_{_{\arbrefo}} + e_{_{\arbrefoo}} ) + \frac {1}{72}\ e_{_{\arbrefp}} + \frac{1}{288}\ ( e_{_{_{_{_{_{_{_{_{\!\!\!\!\!\arbrefq}}}}}}}}}\!\!+ \ e_{_{_{_{_{_{_{_{_{\!\!\!\!\!\!\!\!\!\!\!\!\!\arbrefqq}}}}}}}}} ) \\\\
&\ \ \ \  + \frac{1}{192}\ e_{_{\arbrefr}} \ \ + \frac{1}{576}\ ( \ e_{_{_{_{_{_{_{_{_{_{\!\!\!\!\!\!\!\!\!\!\arbrefs}}}}}}}}}} + e_{_{_{_{_{_{_{_{_{_{\!\!\arbrefss}}}}}}}}}} \ ) +\frac{1}{1728}\ e_{_{\arbreft}} \ \ \ + \frac{B_4}{4!} B^{2}_1\ ( \ e_{_{_{_{_{_{_{_{_{_{_{_{_{_{\arbrefv}}}}}}}}}}}}}} + e_{_{_{_{_{_{_{_{_{_{_{_{_{_{_{_{\!\!\!\arbrefvv}}}}}}}}}}}}}}}} + e_{_{_{_{_{_{_{_{_{_{_{_{_{_{_{_{\!\!\!\!\!\!\!\!\arbrefvvv}}}}}}}}}}}}}}}} + e_{_{_{_{_{_{_{_{_{_{_{_{_{_{\!\!\!\!\!\!\!\!\!\!\!\arbrefvvvv}}}}}}}}}}}}}}\!\!)\\\\
&\ \ \ \  + \frac{B_2}{2!}\frac{B_4}{4!}\ (e_{_{_{_{_{_{_{\ \!\arbrefu}}}}}}} + e_{_{_{_{_{_{_{_{\!\arbrefuu}}}}}}}} + e_{_{_{_{_{_{_{_{\!\!\!\!\!\!\arbrefuuu}}}}}}}} + e_{_{_{_{_{_{_{\!\!\!\!\!\!\!\!\!\!\arbrefuuuu}}}}}}}) + B^{2}_1\frac{B_4}{4!}\ (e_{_{_{_{_{_{_{\!\arbrefw}}}}}}} + e_{_{_{_{_{_{_{_{\!\arbrefww}}}}}}}} + e_{_{_{_{_{_{_{_{\!\arbrefwwww}}}}}}}} + e_{_{_{_{_{_{_{_{\!\!\!\!\arbrefwww}}}}}}}} + e_{_{_{_{_{_{_{_{_{\!\!\!\!\arbrefwwwww}}}}}}}}} + e_{_{_{_{_{_{_{_{_{\!\!\!\!\!\!\!\!\!\arbrefwwwwww}}}}}}}}})\\\\
&\ \ \ \  + \frac{B_6}{6!}\ e_{_{\arbrefx}}\,.       
\end{align*}

Due to the recursive nature of the pre-Lie Magnus expansion at the orders calculated above, and thanks to the pre-Lie identity, we observe that many terms $e_{\tau}$ are omitted in this expansion, for example\label{om3}:
\begin{enumerate}
\item At order four, two terms $e_{\tau}$ out of $4$ can be removed in $\dot{\Omega}_{4}(\!\racine)$, namely $e_{_{_{_{_{\!\arbrecb}}}}},\, e_{_{_{_{\!\arbrecc}}}}$\!\!\!\!.
\item At order five, three terms $e_{\tau}$ out of $10$ can be removed in $\dot{\Omega}_{5}(\!\racine)$, the trees of these omitted terms are:
$$\arbredb,\,\arbredc,\,\arbrede.$$   
\item At order six, the terms of $11$ out of $26$ trees can be removed in $\dot{\Omega}_{6}(\!\racine)$, these trees are:

$$\arbreeddhh,\,\arbreddb,\,\arbredddb,\, \arbreedb,\,\arbreddd,\, \arbredde,\,\arbreedf,\!\arbreedff,\,\arbreedg,\,\arbreeddh,\!\arbreedddh\;.$$
\item At order seven, the terms of $23$ out of $73$ trees can be removed in $\dot{\Omega}_{7}(\!\racine)$.  
\end{enumerate}

\begin{rmk}
This reduction of pre-Lie Magnus expansion terms is not unique, for example, at order five, we can write the formula $\dot{\Omega}_5(\racine)$ with another seven reduced terms, as follows:
$$\dot{\Omega}_{5}(\!\racine) = B_1^2 B_2\frac{3}{2} \,e_{_{_{_{\arbredb}}}} + B_1^2 B_2\frac{3}{2} \,e_{_{\arbredc}} + B_{1}^{2}B_2\,e_{_{\arbrede}} + B_{1}^{2}\frac{B_2}{2!}\,e_{_{\arbredf}} + \frac{B_{2}^{2}}{2!2!}\,(e_{_{\arbredz}}+e_{_{\!\!\!\!\!\!\arbredzz}}\,) + \frac{B_4}{4!}\,e_{_{\arbredh}}\,.$$
\end{rmk}

Now, from the joint works of F. Patras with F. Chapoton \cite{CP13}, and with K. Ebrahimi-Fard \cite{EP14}, recall that: a (non-planar) forest $F = t_1 \cdots t_n$ is a commutative product of (non-planar) rooted trees $t_i$. Denote by $w(F)$ by the number of trees in $F$, which is called the weight of a forest $F$, for example $w(t_1 \cdots t_n) = n$. Let $\Cal{F}$ be the linear span of the set of (non-planar) forests, it forms together with the concatenation an associative commutative algebra. Define another product $\ast$ on $\Cal{F}$ by:
\begin{equation}
(t_1 \cdots t_n) \ast (t'_1 \cdots t'_m) := \sum_f{F_0 (F_1 \to t_1) \cdots (F_n \to t_n),}
\end{equation}
where the sum is over all function $f$ from $\{1, \ldots, m\}$ to $\{0, \ldots, n\}$, and $F_i := \!\!\!\prod\limits_{j\,\in f^{-1}(i)}^{}{\!\!\!t'_j}$. The space $\Cal{F}$ forms an associative non-commutative algebra together with the product $\ast$ defined above. This algebra can be provided with a unit element, sometimes it is the empty tree. This unital algebra is called the Grossman-Larson algebra and denoted by $GL := \Cal{F}$. This algebra acts naturally on $\Cal{T}$\label{cnpnd2} by the extending pre-Lie product $\to$\label{to6}. This action can be defined recursively by:

\begin{equation}
\big(F \ast F'\big) \to t:= F' \to \big(F \to t\big),
\end{equation}    
for any $F, F' \in GL$ and $t$ is a (non-planar) rooted tree.
\begin{exam}
For any $t, t_1, t_2$ (non-planar) rooted trees, we have:
$$(t_1 t_2) \to t = t_2 \to (t_1 \to t) - (t_2 \to t_1) \to t.$$
\end{exam} 

The Grossman-Larson algebra $\big(GL, \ast\big)$ is isomorphic to the enveloping algebra of the underlying Lie algebra of $\big(\Cal{T}, \to\big)$. This construction also works for the enveloping algebra of any pre-Lie algebra \cite{GO08}. We refer the reader to the references \cite{GO08}, \cite{CP13}, \cite{EP14}, for more details about this type of algebras and some of its applications. Hence, the formula of pre-Lie Magnus expansion described in \eqref{logarithm for.} can be rewritten:

\begin{equation}\label{best formula}
\dot{\Omega}(\racine) = log^*(e^{\racine}\,) = \sum_{ n > 0}{\frac{(-1)^{n-1}}{n} ( e^{\racine} - 1)^{\ast\, n-1} \to \racine}\,,
\end{equation}
\noindent where $e^{\racine} = exp(\racine) := \sum\limits_{n \geq 0}^{}{\frac{\racine^{\;n}}{n!}}$, for $F = \racine^{\;n}$ is a forest of one-vertex trees with weight $w(F) =n$, and $\ast$ is the Grossman-Larson product. \\

In fact, we study here the undecorated case, with respect to the forests and trees, of the joint works of F. Patras with F. Chapoton, and with K. Ebrahimi-Fard respectively. The decorated version has been studied in \cite{CP13}, \cite{EP14}.\\

Here, we calculate the few first pre-Lie Magnus elements $\dot{\Omega}_n(\racine)$, up to $n=5$, according to the formula \eqref{best formula} above: 

\begin{align*}
\dot{\Omega}_1(\racine) & = \racine\,.&\\
\dot{\Omega}_2(\racine) & = - \frac{1}{2} \arbrea\, = B_{1}\ e_{_{\arbrea}}.&\\
\dot{\Omega}_3(\racine) & = \frac{1}{3} \arbreba + \frac{1}{12} \arbrebb = B^2_{1}\ e_{_{\arbreba}} + \frac{B_2}{2!}\ e_{_{\arbrebb}}\,.&\\
\dot{\Omega}_4(\racine) & = - \frac{1}{4} \arbreca - \frac{1}{12} \arbrecb - \frac{1}{12} \arbrecc =  \frac {B_1}{3}\ e_{_{\arbreca}} + B_1B_2\ e_{_{\!\!\!\!\!\!\arbreccc}}\,.&\\
\dot{\Omega}_5(\racine) & = \frac{1}{5} \arbreda + \frac{3}{40} \arbredb + \frac{1}{10} \arbredc  + \frac{1}{180} \arbredd + \frac{1}{60} \arbredf + \frac{1}{20} \arbrede + \frac{1}{120} \arbredz - \frac{1}{120} \arbredg - \frac{1}{720} \arbredh &\\
 & = -B_1\frac{B_2}{2!} \frac{5}{2}\ e_{_{\arbreda}}-B_1\frac{B_2}{2!} \frac{1}{2}\ e_{_{\!\!\!\!\!\!\arbredcc}} + B_{1}^{2}B_2\ e_{_{\!\!\!\!\!\!\arbredee}} + B_{1}^{2}\frac{B_2}{2!}\ e_{_{\arbredf}} + \frac{B_{2}^{2}}{2!2!}\ (e_{_{\arbredz}}+e_{_{\!\!\!\!\!\!\arbredzz}}\ ) + \frac{B_4}{4!}\ e_{_{\arbredh}}\,.&
\end{align*}
 
\begin{rmk}
We observe that the formula \eqref{best formula} reduces the number of terms in the pre-Lie Magnus expansion the same way as the reduction induced by the pre-Lie identity in formula \eqref{formula D}. In other words, formula \eqref{best formula} can be considered as a best formula for the reduced pre-Lie Magnus expansion. It would be interesting to have an explanation of this striking fact. 
\end{rmk}

\section[Magnus expansion using a monomial basis for free Lie algebra]{A combinatorial approach  for Magnus expansion using a monomial basis for free Lie algebra}\label{section IIII}

A. Iserles and S. P. No\!\!/\;\!rsett, in their joint work \cite{IN99}, studied the differential equation:
\begin{equation}
\dot{y} = a(t) y , t \geq 0, y(0) = y_0 \in G,
\end{equation}
where $G$\label{G2} is a Lie group, $a \in Lip[\mathbb{R}^+ \rightarrow \Cal{L}]$, the set of all Lipschitz functions\footnote{ A real-valued function $f$ is said to be a Lipschitz function if and only if it satisfies: $|f(x)-f(y)| \leq c|x-y|$, for all $x$ and $y$, where $c$ is a constant independent of $x$ and $y$.} from $\mathbb{R}^+$ into $\Cal{L}$, the Lie algebra of $G$. By considering the Magnus expansion, they have demonstrated, using a numerical method, how to write the Magnus expansion in terms of nested commutators $[a(t_1), [a(t_2), [\ldots, [a(t_{k-1}), a(t_k)]\ldots]]]$ of $a(t_i)$ at different nodes $t_i \in [t_0, t_0+h]$, where $h$ is the time step size. They observed that this numerical method requires the evaluation of  a large number of these commutators, which can be accomplished in tractable manner by exploiting the structure of the Lie algebra.\\  

Different strategies have been developed to reduce the total number of commutators, e.g. the use of so-called time symmetry property\footnote{For more details about this property see \cite{BFJ2000}.} and the concept of a graded free Lie algebra \cite{HM99}. In their joint work \cite{BFJ2000}, the three authors S. Blanes, F. Casas, and J. Ros proposed to apply directly the recurrence of Magnus expansion, described in \eqref{magnus formula}, in numerical version to a Taylor series expansion of the matrix $A(t)$. They reproduced the Magnus expansion terms with a linear combination of nested commutators involving $A$.\\

These authors pursued this strategy with a careful analysis of the different terms of the Magnus expansion by considering its behaviour with respect to the time-symmetry. In the following, we review the part of their work corresponding to their strategy of rewriting Magnus expansion terms, as follows: by taking advantage of the time-symmetry property, they considered a Taylor expansion of $A(t)$ around $t_{\frac{1}{2}} = t_0 + \frac{h}{2}$ as:

\begin{equation}
A(t) = \sum_{i\,\geq 0}{a_i(t - t_{\frac{1}{2}})^i}, \hbox{ where } a_i = \frac{1}{i!} \frac{d^i A(t)}{dt^i}\Big|_{t\,=\,t_{\frac{1}{2}}},
\end{equation}  
and computed the corresponding terms of the component $\Omega_{k}(t_0+h, t_0)$ in the Magnus expansion, where:
$$\Omega_{k} = h^k \!\!\!\!\!\!\sum_{1 \leq i_1, \ldots, i_k \leq N}{\!\!\!\!\!\!\beta_{i_1 \ldots i_k} [A(t_{i_1}), [A(t_{i_2}), [\ldots, [A(t_{i_{k-1}}), A(t_{i_k})]\ldots]]]} + \Cal{O}(h^{2n+1}), \hbox{ for } t_{i_k} \!\!\in [t_0, t_0 +h],$$
by taking into account the linear relations between different nested commutators due to the Jacobi identity. We give here the calculation for the components $\Omega_{k}$, up to $k=6$, obtained by their code \cite[section 3]{BFJ2000}:

\begin{align*}
\Omega_1 &= q_1 + \frac{1}{12} q_3 + \frac{1}{80} q_5 +  \frac{1}{448} q_7\,.&\\\\
\Omega_2 &= \frac{-1}{12} [q_1, q_2]  + \Big(\frac{-1}{80} [q_1, q_4] + \frac{1}{240} [q_2, q_3] \Big) + \Big(\frac{-1}{448} [q_1, q_6] + \frac{1}{2240} [q_2, q_5] - \frac{1}{1344} [q_3, q_4]\Big)\,.&\\\\
\Omega_3 &=  \Big(\frac{1}{360} [q_1, [q_1, q_3]] - \frac{1}{240} [q_2, [q_1, q_2]] \Big) + \Big(\frac{1}{1680} [q_1, [q_1, q_5]] - \frac{1}{2240} [q_1, [q_2, q_4]] + \frac{1}{6720} [q_2, [q_2, q_3]] + &\\
&\;\;\;\;\;\;\frac{1}{6048} [q_3, [q_1, q_3]] -  \frac{1}{840} [q_4, [q_1, q_2]]\Big)\,.&\\\\
\Omega_4 &= \frac{1}{720} [q_1, [q_1, [q_1, q_2]]] + \Big(\frac{1}{6720} [q_1, [q_1, [q_1, q_4]]] - \frac{1}{7560} [q_1, [q_1, [q_2, q_3]]] + \frac{1}{4032} [q_1, [q_3, [q_1, q_2]]] + &\\
&\;\;\;\;\;\;\frac{11}{60480} [q_2, [q_1, [q_1, q_3]]] -  \frac{1}{6720} [q_2, [q_2, [q_1, q_2]]]\Big)\,.&\\\\
\Omega_5 &= \frac{-1}{15120} [q_1, [q_1, [q_1, [q_1, q_3]]]] - \frac{1}{30240} [q_1, [q_1, [q_2, [q_1, q_2]]]]  + \frac{1}{7560} [q_2, [q_1, [q_1, [q_1, q_2]]]] \,.&\\\\
\Omega_6 &= \frac{-1}{30240} [q_1, [q_1, [q_1, [q_1, [q_1, q_2]]]]],
\end{align*}

\noindent where $q_i = a_{i-1} h^i$, for $i \geq 1$, are matrices.\\

The set $E:=\{q_i : i \in \mathbb{N} \}$\label{E40} can be considered as a generating set of a graded free Lie algebra, with $|q_i| = i$ \cite{HM99}. In their computations, S. Blanes, F. Casas, and J. Ros computed the dimensions of the graded free Lie algebra $\Cal{L}(E)$ generated by the set $E$, according to Munthe-Kaas and Owren's work \cite{HM99}. Also, they computed the number of elements of the Lie algebra $\Cal{L}(E)$\label{E41} appearing in the Magnus expansion, when a Taylor series of $A(t)$ around $t =t_0$ and $t = t_{\frac{1}{2}}$ respectively.\\

Here, we review some of their computations as follows: at the order $s=4$, we have $dim_{_{\leq\,4}}\Cal{L} = 7$, with basis elements $q_1, q_2, q_3, q_4, [q_1, q_2], [q_1, q_3], [q_1, [q_1, q_2]]$, such that six of these elements appear in Magnus expansion around $t = t_0$, that are: $q_1, q_2, q_3, q_4, [q_1, q_2], [q_1, q_3]$, with two commutators. Whereas, three elements, $q_1, q_3, [q_1, q_2]$, only appear in Magnus expansion around $t = t_{\frac{1}{2}}$, with one commutator, as it is shown above. For more details about these results see \cite[Section 3, Pages 439-441]{BFJ2000}.

Now, we try to introduce a combinatorial vision of the work above, using the notion of the monomial basis for free Lie algebra $\Cal{L}(E)$, that we obtained in our work \cite{AM20}. Let $\Cal{P\!L}(\racine)$ (respectively $\Cal{P\!L}(E)$) be the free pre-Lie algebra with one generator $\racine$ (respectively generated by the set $\big\{ \racineLabai \; : a_i \in E \big\}$), together with the grafting $\to$\label{to7}. Denote $\widehat{\Cal{P\!L}}(\racine)$ (respectively $\widehat{\Cal{P\!L}}(E)$) by the completion of $\Cal{P\!L}(\racine)$ (respectively $\Cal{P\!L}(E)$) with respect to the filtration given by the degree, which are pre-Lie algebras together with the pre-Lie grafting. Let $a = \sum\limits_{ e \in E}^{}{\lambda_e \racineLabae}$ \;be an element in $\widehat{\Cal{P\!L}}(E)$, that is an infinite linear combination of the generators $\racineLabae\,, e \in E$. \\

Define the map $G_a: \Cal{P\!L}(\racine) \rightarrow \widehat{\Cal{P\!L}}(E)$ to be the unique pre-Lie homomorphism that is induced by the universal property of the freeness of $\Cal{P\!L}(\racine)$:

\begin{figure}[h]
\centering
\diagrama{
\xymatrix{
\{\racine\}  \ar@{^(_->}[r]^i \ar@{->}[dr]_{f }&  \Cal{P\!L}(\racine) \ar@{->}[d]^{G_a}\\
 & \widehat{\Cal{P\!L}}(E)}}
\label{}
\caption{}
\end{figure}

\noindent such that $G_a(\racine) = a$. 

\begin{lem}\label{g. result}
For any (undecorated) planar rooted tree $\tau$, we have:
\begin{equation}\label{another rewriting}
G_a\big(\overline{\Psi}(\tau) \big) = \sum_{\delta : V(\tau) \rightarrow E}{\Big(\prod\limits_{v \in V(\tau)}^{}{\lambda_{\delta(v)}}\Big)\; \overline{\Psi}(\tau_{\delta})},
\end{equation}
where $\overline{\Psi}: \Cal{T}^{E}_{\!\!pl} \rightarrow \Cal{T}^{E}$\label{cnp6}, in the right hand side, is described in \cite[Subsection $2.2$]{AM2014} (we use the same letter for the undecorated version from $\Cal{T}_{\!\!pl}$ onto $\Cal{T}$), and where $\tau_{\delta} \in  T^{E}_{\!\!pl}$ is the tree $\tau$ decorated according to the map $\delta$.    
\end{lem}
\begin{proof}
Let $\tau$ be any (undecorated) planar rooted tree, we have that $\overline{\Psi}(\tau)= m(\racine, \to)$ is a monomial, in $\Cal{P\!L}(\racine)$, of the one-vertex tree $\racine$ multiplied (by itself) using the pre-Lie product $\to$. From the definition of $G_a$ above, we get:
\begin{equation}\label{def. of Ga}
G_a\big(\overline{\Psi}(\tau)\big) = G_a\big(m(\racine, \to)\big) = m(a, \to), 
\end{equation}
where $m(a, \to)$ is the monomial of $a$, in $\widehat{\Cal{P\!L}}(E)$, induced from the monomial $m(\racine, \to)$ by sending the one-vertex tree into its image $G_a(\racine)=a$.\\

We proceed by induction on the number $n$ of vertices, the case $n = 1$ being obvious. Suppose that the formula \eqref{another rewriting} is true up to $n-1$ vertices. Let $\tau \in T^{n}_{\!pl}$, we have that $\tau$ can be written in a unique way as $\tau = \tau_1 \lbutcher \tau_2$, hence\label{lb10}:
\begin{align*}
G_a\big(\overline{\Psi}(\tau) \big) & = G_a\big(\overline{\Psi}(\tau_1 \lbutcher \tau_2) \big) &\\
& = G_a\big(\overline{\Psi}(\tau_1) \to \overline{\Psi}(\tau_2)\big)&\\
& = G_a\big(\overline{\Psi}(\tau_1)\big) \to  G_a\big(\overline{\Psi}(\tau_2)\big)&\\
& = \sum_{{\delta_1 : V(\tau_1) \rightarrow E} \atop {\delta_2 : V(\tau_2) \rightarrow E}}{\Big( \prod\limits_{v \in V(\tau_1)}^{}{\lambda_{\delta_1(v)}} \; \prod\limits_{v' \in V(\tau_2)}^{}{\lambda_{\delta_2(v')}} \Big)\; \overline{\Psi}(\tau_{_{\!1,\,\delta_{1}}}) \to \overline{\Psi}(\tau_{_{\!2,\,\delta_{2}}})}&\\
& = \sum_{ \delta : V(\tau) \rightarrow E}{\Big( \prod\limits_{v \in V(\tau)}^{}{\lambda_{\delta(v)}} \Big)\;\overline{\Psi}(\tau_{\delta})}, \hbox{ (by setting } \tau_{\delta} = \tau_{_{\!1,\,\delta_{1}}} \!\!\lbutcher \tau_{_{\!2,\,\delta_{2}}}).&
\end{align*}  
\end{proof}

\begin{lem}\label{PL(E)- Magnus element}
The pre-Lie Magnus element $\dot{\Omega}(a)$ in $\widehat{\Cal{P\!L}}(E)$\label{E42} can be rewritten as:
\begin{equation}
\dot{\Omega}(a) = \sum_{\tau \in T^{e1}_{\!pl}}{\gamma(\tau)\;G_a\big(\overline{\Psi}(\tau)\big)}, 
\end{equation}
where $a = \sum\limits_{ e \in E}^{}{\lambda_e \racineLabae} \in \widehat{\Cal{P\!L}}(E)$.
\end{lem}
\begin{proof}
From Theorem \ref{NRT} and Lemma \ref{rel. between F and Psi}, we have that:
\begin{equation}
\dot{\Omega}(\racine) = \sum_{\tau \in T^{e1}_{\!pl}}{\gamma(\tau)\;\overline{\Psi}(\tau)}.
\end{equation} 
We have that $\dot{\Omega}(\racine)$ is an element in $\widehat{\Cal{P\!L}}(\racine)$, and the map $G_a$ can be extended linearly from $\widehat{\Cal{P\!L}}(\racine)$ into $\widehat{\Cal{P\!L}}(E)$, such that:
$$\dot{\Omega}(a) := G_a \Big(\dot{\Omega}(\racine)\Big) =\sum_{\tau \in T^{e1}_{\!pl}}{\gamma(\tau)\;G_a\big(\overline{\Psi}(\tau)\big)}.$$
This proves the Lemma.
\end{proof}

In Lemma \ref{g. result} above, let us denote $\lambda(\tau_{\delta}) := \prod\limits_{v \in V(\tau)}^{}{\lambda_{\delta(v)}}$. Hence, we can simplify the formula \eqref{another rewriting} as:

\begin{equation}\label{obtained}
G_a\big(\overline{\Psi}(\tau) \big) = \!\!\!\!\sum_{\delta : V(\tau) \rightarrow E}{\lambda{(\tau_{\delta})} \; \overline{\Psi}(\tau_{\delta})}.
\end{equation}

\noindent Consequently, we can get the following result.

\begin{prop}
The pre-Lie Magnus expansion can be rewritten: 
\begin{equation}\label{generalization case}
\dot{\Omega}(a) = \sum_{\sigma\,\in T^{E,\,e1}_{\!pl}}{\!\!\!\gamma(\sigma)\; \lambda(\sigma)\,\overline{\Psi}(\sigma)},
\end{equation}
for any $\sigma\,\in T^{E,\,e1}_{\!pl}$. Here $\gamma : \Cal{T}^E_{\!\!pl} \rightarrow K$\label{e13} defined as in \eqref{definition of alph.}, forgetting the decoration. 
\end{prop}
\begin{proof}
From Lemma \ref{PL(E)- Magnus element}, and by substituting $G_a\big(\overline{\Psi}(\tau) \big)$ obtained in \eqref{obtained}, we get:
$$\dot{\Omega}(a) = \sum_{{\tau\,\in T^{e1}_{\!pl}} \atop {\delta : V(\tau) \rightarrow E}}{\!\!\!\!\gamma(\tau)\; \lambda(\tau_{\delta})\,\overline{\Psi}(\tau_{\delta})} = \sum_{\sigma\,\in T^{E,\,e1}_{\!pl}}{\!\!\!\gamma(\sigma)\; \lambda(\sigma)\,\overline{\Psi}(\sigma)}.$$
This proves the Proposition.
\end{proof}

\begin{rmk}
The formula for the pre-Lie Magnus expansion in \eqref{generalization case} can be considered as a generalization of the formula \eqref{KM formula}. In other words, it is a decorated version of \eqref{KM formula}, taking into account the relation between the maps $F$ and $\overline{\Psi}$ described in Lemma \ref{rel. between F and Psi}. 
\end{rmk}

The pre-Lie homomorphism $\Phi : \big(\Cal{P\!L}(E), \to \big)  \longrightarrow \big( \Cal{L}(E), \rhd \big)$\label{E43}, described in our work \cite[Section $4$]{AM20}, respects the degree, it is then  continuous for the topologies defined by the corresponding decreasing filtrations\footnote{These topologies are induced by metrics defined on pre-Lie algebra using compatible decreasing filtrations described in \cite{NJ80, AR81, D.M}.}. We denote by the same letter $\Phi$ the pre-Lie homomorphism from the completed pre-Lie algebra $\widehat{\Cal{P\!L}}(E)$ onto $\widehat{\Cal{L}}(E)$:

\begin{figure}[h]
\centering
\diagrama{
\xymatrix{
\Cal{P\!L}(E) \ar@{^(_->}[r]^i \ar@{->>}[d]_{\Phi}&  \widehat{\Cal{P\!L}}(E) \ar@{->>}[d]^{\Phi}\\
\Cal{L}(E) \ar@{^(_->}[r]^i &  \widehat{\Cal{L}}(E)}}
\caption{}
\label{}
\end{figure}

\noindent We can get another representation of pre-Lie Magnus expansion, as in the following result.

\begin{cor}
The pre-Lie Magnus expansion in $\widehat{\Cal{L}}(E)$ can be rewritten as:
\begin{equation}\label{L(E)- Magnus element}
\dot{\Omega}(x) = \sum_{\sigma\,\in T^{E,\,e1}_{\!pl}}{\!\!\!\gamma(\sigma)\; \lambda(\sigma)\,\Phi\Big(\overline{\Psi}(\sigma)\Big)},
\end{equation}
where $x = \Phi(a) = \sum\limits_{e \in E}^{}{ \lambda_e e } \in \widehat{\Cal{L}}(E)$, for $e = \Phi(\racineLabae) \in E$\label{e14}.
\end{cor}

As a particular case, let us take $E = \bigsqcup\limits_{i \in \mathbb{N}}^{}{E_i}$,  with $\# E_i = 1$, for all $i \in \mathbb{N}$, i.e. $E = \{ a_i : i \in \mathbb{N} \}$, such that $|a_i| = i$, and the generators are ordered by:
$$ a_1 < a_2 < \cdots < a_s < \cdots .$$

For any $\sigma\,\in T^{E,\,e1}_{\!pl}$, $\Phi\Big(\overline{\Psi}(\sigma)\Big)$ is an element in $\Cal{L}(E)$. From our work in \cite[section $6$]{AM20}, we have that the set $\widetilde{\Cal{B}} = \big\{ \Phi(t) : t \in O(I) \big\}$ forms a monomial basis for the pre-Lie algebra $\big( \Cal{L}(E), \rhd \big)$ (respectively for the free Lie algebra $\big( \Cal{L}(E), [\cdot, \cdot] \big)$), where the pre-Lie product $\rhd$ is defined by:
\begin{equation}
x \rhd y := \frac{1}{|x|} [x, y],
\end{equation}
for $x, y \in \Cal{L}(E)$, hence:

$$\Phi\Big(\overline{\Psi}(\sigma)\Big) = \alpha_{1} \Phi(t_1) +  \alpha_{2} \Phi(t_2) + \cdots + \alpha_{k} \Phi(t_k) ,$$
is a linear combination of basis elements $\Phi(t_i),\,t_i\!\!\in\!\!O(I)$, multiplied by coefficients $\alpha_i\!\!\in\!\!K$, for all $i=1, \ldots, k$, where $I$ is the (two-sided) ideal of $\Cal{T}^E$ generated by all elements on the form:
$$|s| (s \to t ) + |t| (t \to s), \hbox{ for } s, t \in \Cal{T}^E.$$
Thus, the pre-Lie Magnus expansion in \eqref{L(E)- Magnus element} can be expressed using the monomial basis elements $\Phi(t)$, for $t \in O(I)$\label{Tf10}. Here, we calculate the few first reduced pre-Lie Magnus elements $\dot{\Omega}_n(x)$ in $\widehat{\Cal{L}}(E)$, up to $n=5$:

\begin{align*}
\dot{\Omega}_1(x) & = \lambda_1 a_1.&\\
\dot{\Omega}_2(x) & = \lambda_2 a_2.&\\
\dot{\Omega}_3(x) & = \lambda_3 a_3 - B_1^2 \lambda_1 \lambda_2\,a_1 \rhd a_2.&\\
\dot{\Omega}_4(x) & = \lambda_4 a_4 + B_1\frac{2}{3} \lambda_1 \lambda_3\,a_1 \rhd a_3 + B_1^2 \frac{1}{2} \lambda_1^2 \lambda_2\,(a_1 \rhd a_2) \rhd a_1.&\\
\dot{\Omega}_5(x) & = \lambda_5 a_5 + B_1 \big( \frac{3}{4} \lambda_1 \lambda_4\,a_1 \rhd a_4 + \frac{1}{3} \lambda_2 \lambda_3\,a_2 \rhd a_3 \big)+ B_1 \frac{5}{9} \lambda_1^2 \lambda_3\,(a_1 \rhd a_3) \rhd a_1 + B_1^2 \frac{11}{12} \lambda_1^3 \lambda_2\,\big((a_1 \rhd a_2) \rhd a_1\big) \rhd a_1,&
\end{align*}

\noindent and using $ a_i \rhd a_j = \frac{1}{|a_i|} [a_i, a_j]$, for all $i, j$, we get:

\begin{align*}
\dot{\Omega}_1(x) & = \lambda_1 a_1.&\\
\dot{\Omega}_2(x) & = \lambda_2 a_2.&\\
\dot{\Omega}_3(x) & = \lambda_3 a_3 - B_1^2 \lambda_1 \lambda_2\, [a_1, a_2].&\\
\dot{\Omega}_4(x) & = \lambda_4 a_4 + B_1\frac{2}{3} \lambda_1 \lambda_3\, [a_1, a_3] + B_1^2 \frac{1}{6} \lambda_1^2 \lambda_2\, [[a_1, a_2], a_1].&\\
\dot{\Omega}_5(x) & = \lambda_5 a_5 + B_1 \big( \frac{3}{4} \lambda_1 \lambda_4\,[a_1, a_4] + \frac{1}{6} \lambda_2 \lambda_3\, [a_2, a_3] \big)+ B_1 \frac{5}{36} \lambda_1^2 \lambda_3\, [[a_1, a_3], a_1] + B_1^2 \frac{11}{144} \lambda_1^3 \lambda_2\,\big[[[a_1, a_2], a_1], a_1\big].&
\end{align*}

Here, we link between our work in \cite[Section $4$]{AM20}, on the pre-Lie construction of the Lie algebras, and the work of  S. Blanes, F. Casas and J. Ros \cite{BFJ2000}, on the writing of Magnus expansion. Firstly, we shall consider the generators $\{q_i : i \geq 1\}$, of the Lie algebra $\Cal{L}(E)$ in their work, as matrix-valued functions in $h$. Define a pre-Lie product on the set of formal power series $h \mathbb{R}[[h]]$ by:
\begin{equation}\label{def. I}
(f \rhd g)(h) := [\!\int\limits_{0}^{h}{\frac{f(s)}{s} ds}, g(h)], \hbox{ for any } f, g \in h \mathbb{R}[[h]].
\end{equation}
This pre-Lie product described in \eqref{def. I} can be visualized as in the following diagram:

\begin{figure}[h]
\centering
\diagrama{
\xymatrix{
h \mathbb{R}[[h]] \otimes h \mathbb{R}[[h]] \ar@{->}[r]^{\hspace{25pt}\rhd}\ar@{->}[d]_{\frac{1}{h} \otimes \frac{1}{h}}&  h \mathbb{R}[[h]] \ar@{<-}[d]^{h}\\
\mathbb{R}[[h]] \otimes \mathbb{R}[[h]]  \ar@{->}[r]^{\hspace{25pt}\widetilde{\rhd}}&  \mathbb{R}[[h]]}}
\caption{The description of $\rhd$.}
\label{}
\end{figure}
\noindent where $f\widetilde{\rhd} g(h)=[\int\limits_0^h{f(s)ds}, g(h)]$. Hence, for $q_i(h) = a_{i-1} h^i, q_j(h) = a_{j-1} h^j$ any two generators of $\Cal{L}(E)$\label{Tf11}, we can apply the pre-Lie product defined above in \eqref{def. I} as follows:
\begin{align*}
(q_i \rhd q_j)(h) & = [\!\int\limits_{0}^{h}{\frac{q_i(s)}{s} ds}, q_j(h)] &\\
& = [ a_{i-1} \!\int\limits_{0}^{h}{s^{i-1} ds}, q_j(h)]&\\
& = [ \frac{1}{i} a_{i-1} {h^{i}}, q_j(h)]&\\
& = \frac{1}{i} [q_{i}, q_j](h),&\\
\end{align*}
where $|q_i| = i$, for $i \geq 1$. Simply, we shall write $q_i \rhd q_j = \frac{1}{|q_i|} [q_{i}, q_j]$, for all $i, j \geq 1$. In following, we rewrite the calculations of the three authors for the components $\Omega_k$ up to $k=6$, using the pre-Lie product defined above:

\begin{align*}
\Omega_1 &= q_1 + \frac{1}{12} q_3 + \frac{1}{80} q_5 +  \frac{1}{448} q_7\,.&\\\\
\Omega_2 &= \frac{-1}{12} (q_1 \rhd q_2)  + \Big(\frac{-1}{80} (q_1 \rhd q_4) + \frac{1}{120} (q_2 \rhd q_3) \Big) + \Big(\frac{-1}{448} (q_1 \rhd q_6) + \frac{1}{1120} (q_2 \rhd q_5) - \frac{1}{448} (q_3 \rhd q_4)\Big)\,.&\\\\
\Omega_3 &=  \Big(\frac{1}{360} (q_1 \rhd (q_1 \rhd q_3)) - \frac{1}{120} (q_2 \rhd (q_1 \rhd q_2)) \Big) + \Big(\frac{1}{1680} (q_1 \rhd (q_1 \rhd q_5)) - \frac{1}{1120} (q_1 \rhd (q_2 \rhd q_4)) +  &\\
&\;\;\;\;\;\;\frac{1}{1680} (q_2 \rhd (q_2 \rhd q_3)) + \frac{1}{2016} (q_3 \rhd (q_1 \rhd q_3)) -  \frac{1}{210} (q_4 \rhd (q_1 \rhd q_2)) \Big)\,.&\\\\
\Omega_4 &= \frac{1}{720} (q_1 \rhd (q_1 \rhd (q_1 \rhd q_2))) + \Big(\frac{1}{6720} (q_1 \rhd (q_1 \rhd (q_1 \rhd q_4))) - \frac{1}{3780} (q_1 \rhd (q_1 \rhd (q_2 \rhd q_3))) + \frac{1}{1344} &\\
&\;\;\;\;\;\;(q_1 \rhd (q_3 \rhd (q_1 \rhd q_2))) + \frac{11}{30240} (q_2 \rhd (q_1 \rhd (q_1 \rhd q_3))) -  \frac{1}{1680} (q_2 \rhd  (q_2 \rhd (q_1 \rhd q_2))) \Big)\,.&
\end{align*}
\begin{align*}
\Omega_5 &= \frac{-1}{15120} (q_1 \rhd (q_1 \rhd (q_1 \rhd (q_1, q_3)))) - \frac{1}{15120} (q_1 \rhd (q_1 \rhd (q_2 \rhd (q_1 \rhd q_2)))) + \frac{1}{3780}&\\
&\;\;\;\;\;\;(q_2 \rhd (q_1 \rhd (q_1 \rhd (q_1 \rhd q_2)))) \,.&\\\\
\Omega_6 &= \frac{-1}{30240} (q_1 \rhd (q_1 \rhd (q_1 \rhd (q_1 \rhd (q_1 \rhd q_2))))).
\end{align*}\\\\

\paragraph{\textbf{Acknowledgments}}{I would like to thank D. Manchon for his continuous support to me throughout the writing of this paper. I also thank F. Patras and K. Ebrahimi-Fard for valuable discussions and comments.}

\newpage

%%%%%%%%%%%%%%%%%%%%%%%%%%%%%%%%%%%%%%%%%%%%%%%%%%%%%%%%%%%%%%%%%%%
%%%%%%%%%%%%%%%%%%%%%%%%%%%%%%%%%%%%%%%%%%%%%%%%%%%%%%%%%%%%%%%%%%%


\begin{thebibliography}{abcdsfgh}
{\small{
\bibitem{AR81}A. Agrachev, R. Gamkrelidze, {\sl Chronological algebras and nonstationary vector fields\/}, J. Sov. Math. 17 Nol, 1650-1675 (1981).
\bibitem{AM2014}Mahdi J. Hasan Al-Kaabi, {\sl Monomial Bases for Free Pre-Lie Algebras\/}, S\'eminaire Lotharingien de Combinatoire 71 (2014), Article B71b.
\bibitem{AM20}Mahdi J. Hasan Al-Kaabi, D. Manchon, F. Patras, {\sl Monomial Bases and pre-Lie Structures for Free Lie Algebras\/}, preprint.
\bibitem{BFJ2000}S. Blanes, F. Casas, J. Ros, {\sl Improved High Order Integrators based on the Magnus Expansion\/}, BIT Numerical Mathematics Vol. 40, No. 3, 434-450 (2000).
\bibitem{BCR09}S. Blanes, F. Casas, J. A. Oteo, J. Ros, {\sl The Magnus expansion and some of its applications\/}, Physics Reports 470, 151-238 (2009).
\bibitem{CB00}Ch. Brouder, {\sl Runge-Kutta methods and renormalization\/}, Europ. Phys. J. C12, 521-534 (2000).
\bibitem{CP13}F. Chapoton, F. Patras, {\sl Enveloping algebras of pre-Lie algebras, Solomon Idempotents and the Magnus Formula\/}, Internat. J. of Algebra and Computation Vol. 23, Issue 04 (2013). 
\bibitem{KM09}K. Ebrahimi-Fard, D. Manchon, {\sl A Magnus- and Fer-type Formula in Dendriform Algebras\/}, Foundations of Computational Mathematics, Volume 9, Issue3, 295 (2009).
\bibitem{K.M09}K. Ebrahimi-Fard, D. Manchon, {\sl Dendriform Equations\/}, J. of Algebra 322, 4053-4079 (2009).
\bibitem{EP14}K. Ebrahimi-Fard, F. Patras, {\sl The pre-Lie structure of the time-ordered exponential \/}, Letters in Mathematical Physics, Vol. 104, Issue 10, 1281-1302 (2014).
\bibitem{GO08}D. Guin, J.-M. Oudom, {\sl On the Lie enveloping algebra of a pre-Lie algebra\/}, $K$- theory, 2 (1), 147-167 (2008).
\bibitem{IN99}A. Iserles, S. P. No\!\!/\;\!rsett, {\sl On the solution of linear differential equations in Lie groups\/}, R. Soc. Lond. Philos. Trans. Ser. A Math. Phys. Eng. Sci. 357, 983-1020 (1999).
\bibitem{NJ80}N. Jacobson, {\sl Basic Algebra II: Second Edition\/}, W. H. freeman, USA (1980).
\bibitem{WM54}W. Magnus, {\sl On the Exponential Solution of Differential Equations for a Linear Operator\/}, Commun. Pure and Appl. Math. 7, 649-673 (1954).
\bibitem{D.M}D. Manchon, {\sl Algebraic Background for Numerical Methods, Control Theory and Renormalization\/}, Proc. Combinatorics and Control, Benasque, Spain, 2010, arXiv:1501.07205.
\bibitem{HM99}H. Munthe-Kaas, B. Owren, {\sl Computations in a free Lie algebra\/}, Philos. Trans. Royal Soc. A 357, pp. 957-981 (1999).
\bibitem{AS62}A. I. Shirshov, {\sl Some Algorithmic Problems for Lie Algebras \/}, Siberian Math. J. 3, 292-296 (1962).
\bibitem{S}N. J. A. Sloane, {\sl The On-Line Encyclopedia of Integer Sequences\/}, oeis.org.
}}
\end{thebibliography}
\end{document}